\documentclass{article}
\usepackage[english]{babel}
\usepackage[T1]{fontenc}

\usepackage{lmodern}						
\usepackage{amsmath,amsfonts}				
\usepackage[dvips]{graphicx}                       
\usepackage{epsf,latexsym,amssymb}
\usepackage{mathrsfs}
\usepackage{amsthm}
\usepackage{amscd}
\usepackage{pstricks-add}					
\usepackage{makeidx}

\newtheorem{teo}{Theorem}[section]
\newtheorem{lem}[teo]{Lemma}
\newtheorem{prop}[teo]{Proposition}
\newtheorem{defi}[teo]{Definition}

\newtheorem{cor}[teo]{Corollary}
\newtheorem{oss}[teo]{Remark}

\newcommand{\E}{\widetilde{Ein}}
\newcommand{\M}{\widetilde{M}}
\newcommand{\Z}{\mathbb{Z}}
\newcommand{\R}{\mathbb{R}}
\newcommand{\N}{\mathbb{N}}
\newcommand{\Sf}{\mathbb{S}}

\begin{document}

\title{Maximal extension of conformally flat globally hyperbolic space-times}
\author{Clara Rossi Salvemini}
\maketitle

\abstract{The notion of maximal extension of a globally hyperbolic space-time arises from the notion of maximal  solutions of the Cauchy problem associated to the Einstein's equations of general relativity. Cho\-quet-Bruhat and Geroch proved (\cite{Choquet-Bruhat}) that if the Cauchy problem has a local solution, this solution has a unique maximal extension.  Since the causal structure of a space-time is  invariant under conformal changes of metrics we may generalize this notion of maximality to the conformal setting.  In this article we focus on conformally flat space-times of dimension greater or equal than $3$. In this case, by a Lorentzian version of Liouville's theorem, these space-times are locally modeled on the Einstein space-time. In the first part we  use this fact to prove the existence and uniqueness of the maximum extension for globally hyperbolic conformally flat space-times. In the second  part we find a causal characterization of globally hyperbolic conformally flat maximal space-times whose developing map is a global diffeomorphism. }

\tableofcontents

\section{Introduction}

\paragraph{Cauchy problem associated to the Einstein equation.}
The concept of maximal extension of a globally hyperbolic space-time comes from a PDE problem: the existence and uniqueness of maximal solutions for the Cauchy problem associated to the Einstein equation.
The Einstein equation relates a physical object, the stress-energy tensor, with a geometric one,  the curvature tensor of the universe. We can write it as following:
\begin{eqnarray}\label{eq.Ein}
Ric(g)-1/2 scal(g) g +\Lambda g =8\pi T
\end{eqnarray}
\noindent where $g$ is a Lorentzian metric, $Ric(g)$ is the Ricci tensor, $scal(g)$ is the scalar curvature, $\Lambda$ the cosmological constant and  $T$ the stress-energy tensor which is a symmetric tensor of type $(2,0)$.\\
In vacuum  the tensor $T$ is zero and the equation becomes: $Ric(g)=0$.
 A solution for this equation is just a Lorentzian manifold with Ricci curvature zero.
In the general case, the meaning of what constitute a solution is not clear, because the topology of the universe and the stress-energy tensor are not defined a priori.
A possible strategy to find solutions of the equation (\ref{eq.Ein}) is to assume that the solution is globally hyperbolic. By Geroch's Theorem ( \cite{GerochTimeFunction} and  \cite[p. 1155]{SachsWuTimeFunction1977}) every globally hyperbolic space-time is diffeomorphic to a product $S\times \mathbb{R}$, so that every slice $S\times \{t\}$ is a spacelike submanifold. Then we can formulate a Cauchy problem associated to the equation (\ref{eq.Ein}) as follows. The initial data
 is a Riemannian manifold $(S,h)$ of  dimension $n$ equipped with a symmetric $(2,0)$-tensor $II$, and a solution is a Lorentzian metric $g$ over the product manifold $M:=S\times \mathbb{R}$ such that $g$  verifies the equation  ($\ref{eq.Ein}$) for a tensor $T$ given \em a priori \em on $M$ and $II$ is the  shape tensor of the sub-manifold  $S\times \{0\}$ of $M$. \\
It turns out that a necessary condition to have a solution is that $h$ and $II$ verify some equations,  named the {constraint equations} of general relativity (\cite{Hawking}, ch.$7$ ) when $T=0$.
Geroch and Choquet-Bruhat   proved (in \cite{C-Blocale} ) that , when $T=0$, the constraint equations are also a sufficient condition to the existence and unicity  of local solutions of the Cauchy problem. \\
One may ask   how the solutions develop  far away from the initial data. Is it possible to have different developments out of a neighborhood of the initial data? We say that a solution $M$ \em  extends \em another $N$, if $N$ is isometric to a neighborhood of the initial data in $M$. A \em maximal solution \em is then a solution which has only trivial extensions. In \cite{Choquet-Bruhat} Choquet-Bruhat et Geroch have proved:

\begin{teo}\label{C-B extensione globale}
Any local solution $(M,g)$ of the Cauchy's problem has a maximal extension, which is unique up to isometry.
\end{teo}

\paragraph{Maximal extension into a given category.}\label{section:category}
It is well known that Theorem \ref{C-B extensione globale} naturally generalizes to larger families of space-times  which are not necessarily solutions of the Einstein equations.   Let start by  a definition of maximality  for globally hyperbolic space-times:

\begin{defi}\label{Cauchy-plongemens iso}
Let $M$ and $N$ be two globally hyperbolic space-times.
An isometric embedding $f: M\rightarrow N$ is a \em Cauchy-embedding \em if there exists a Cauchy hypersurface  $S\subset M$ such that $f(S)$ is a Cauchy hypersurface of $N$. In this case we say that $N$ \em extends \em $M$.\\
A globally hyperbolic space-time $M$ is \em maximal \em if every  Cauchy-embedding of $M$ into an other space-time is onto.
\end{defi}

 This more general notion of maximality coincides with the classical one in the case of space-times which are solutions of the same Cauchy problem. Therefore now the problem of existence and uniqueness of the maximal extension of a given globally hyperbolic space-time is well-defined even for space-times which are not solutions of the Cauchy problem. \\
The arguments involved in the proof of the existence of the maximal extension in Theorem \ref{C-B extensione globale} easily generalize: every  globally hyperbolic space-time admits a maximal extension, but in general there  is no reason  for this extension to be unique. 
However the maximal extension is unique if we consider "rigid categories" of space-times:
\begin{defi}\label{categoria}
\emph{A \em category of space-times \em is a  family $\mathscr{F}$ of space-times such that:}
 \begin{itemize}
\item $\mathscr{F}$ is stable by isometry: if $(M,g)$ is in  $\mathscr{F}$ and $(N,h)$ is isometric to $(M,g)$ then $(N,h)$ is in $\mathscr{F}$.\\
\item $\mathscr{F}$ is stable by restriction: if $(M,g)$ is in $\mathscr{F}$ then for every open set $U$ of $M$, $(U,g\vert_U)$ is in $\mathscr{F}$.\\
\item  $\mathscr{F}$ is stable by gluing: if there is an open covering $(U_i)_{i \in I}$ of $(M,g)$ such that for every $i$ of $I$ the restriction $(U_i,g\vert_{U_i})$ is in  $\mathscr{F}$ then $(M,g)$ is in $\mathscr{F}$.
\end{itemize}
\end{defi}

\begin{defi}
A space-time $M$ in a category $\mathscr{F}$ is \em $\mathscr{F}$-maximal \em if  every Cauchy-embedding of $M$ into another space-time of the same category $\mathscr{F}$ is onto.
\end{defi}

\noindent Again, the arguments for the existence in Theorem \ref{C-B extensione globale} apply: every space-time in a category  $\mathscr{F}$ always has  a $\mathscr{F}$-maximal  extension. The uniqueness  comes from some additional hypothesis:

\begin{defi}\label{def: categoria rigida}
A category $\mathscr{C}$  of  space-times is \em rigid  \em  if given two globally hyperbolic space-times $M$  and $N$  in $\mathscr{C}$,  and an isometry $f: I^\pm(p)\to I^\pm(q)$, where  $p\in M$ and $q\in N$,  then $f$ extends into an isometry $\hat{f}: U\cup I^\pm(p)\to V\cup I^\pm(q)$, where $U$  and $V$  are neighborhoods of $p$  and $q$.
\end{defi}

 This is the key property used in the proof of the theorem  of  Choquet-Bruhat and Geroch, who considered the category of space-times which are solutions of the same Cauchy problem; one of the steps of the proof is to show that this category is rigid.\\  
Another important rigid category of space-times is the $(G,X)$-category, where $X$ is a fixed space-time and   $G$ its isometry group. The elements of this category  are space-times which are $(G,X)$-manifolds. It's easy to see that it is rigid: let $M$ and  $N$ two such space-times and let $f: I^-(p)\to I^-(q)$ an isometry such that $f(p)=q$ and let $\phi: U\to X$ and $\psi:V\to X$ two charts on the neighborhood of $p$ and $q$. By definition of $(G,X)$-manifolds, the isometry 
$$\psi\circ f\circ\psi^{-1}: \psi(V\cap I^-(q) )\to \psi(U\cap I^-(p))$$
 extends into a unique element  $g\in G$. Then the map   $\hat{f}:=\psi\circ g\circ\phi $ is an isometry between $U$ and $V$.\\
Space-times of constant curvature are examples of $(G,X)$-manifolds:     $X$ is  the Minkowski space-time $\R^{1,n}$, when the curvature is zero,  the de Sitter space-time $d\Sf_{1,n}$, when it is positive, and  the anti-de Sitter space-time $Ad\Sf_{1,n}$, when it is negative. Then we have notions of $\R^{1,n}$-maximal extension,  $d\Sf_{1,n}$-maximal  extension, and   $Ad\Sf_{1,n}$-maximal  extension. \\
Since constant curvature space-times  are  solutions of the Einstein equation, by Theorem \ref{C-B extensione globale}, these extensions are unique up to isometry. In fact, Theorem \ref{C-B extensione globale} is true for every rigid category:

\begin{teo}\label{teo: massimalità categiria rigida}
Every globally hyperbolic space-time in a rigid category $\mathscr{C}$ has a unique $\mathscr{C}$-maximal  extension.
\end{teo}

 This statement is quite well-known by the experts of the field; in the present paper we will consider a slightly different problem, where we consider extensions by conformal embeddings, not necessarily isometric (see the next section).
The tools and proof involved in this new context can be easily adapted to the isometric case, providing a complete
proof of Theorem \ref{teo: massimalità categiria rigida}.



\paragraph{Maximal conformally flat extension.}
Since the causal structure is a conformal invariant, the notion of maximality defined in the previous section naturally generalizes to conformal classes of Lorentzian metrics. This is obtained by taking \em conformal Cauchy-embeddings\em , instead of isometric ones,  in Definition \ref{Cauchy-plongemens iso}.   Then we  say that a space-time $M$ is \em $C$-maximal \em if every conformal Cauchy-embedding of $M$ into another globally hyperbolic space-time is onto.\\
Even if here we are in the conformal context we can still use the language of category by just replacing the word "isometry" with "conformal diffeomorphism" in Definition \ref{categoria}. However, just as in the isometric case, the proof of the existence and uniqueness of the maximal extension requires some additional rigidity property.
Moreover the family of conformally flat space-times is a sub-category of the $C$-category: we call it the  $C_0$-category.  So we can have a well-defined notion of $C_0$-maximality:
a conformally flat space-time $M$ is \em $C_0$-maximal\em , if every conformal Cauchy-embedding of $M$ into another  globally hyperbolic conformally flat space-time in onto.\\
Now that we have defined the notion of $C$-maximality and $C_0$-maximality, we can again ask the questions: does every conformally flat globally hyperbolic space-time have a $C$-maximal and $C_0$-maximal extension? Are these extensions unique up to conformal diffeomorphisms?
The answer is the following generalization of Theorem \ref{C-B extensione globale} to the conformal and conformally flat case. 

\begin{teo}\label{teo max}
Every  globally hyperbolic conformally flat space-time $M$ of dimension $\geq 3$ has a unique $C_0$-maximal extension. This extension is unique up to conformal diffeomorphism.
\end{teo}

 In the  section $ \ref{sec: C0 extension}$ of this article we give a proof of this result, using the fact that  conformally flat space-times are $(G,X)$-manifolds, where  $X$ is the Einstein space-time and $G$ its conformal group of diffeomorphisms.  This gives an additional rigidity property:  if we define the \em conformally rigid category \em by  taking conformal diffeomorphisms, instead of isometries,  in Definition \ref{def: categoria rigida}, the $C_0$-category is conformally rigid. \\
A generalization of Theorem \ref{teo max} to the  $C$-category seems much more difficult to prove. We do not think that the techniques used here are enough to deal with the $C$-category. We do not know for the moment if a  $C$-maximal extension does exist for every space-time, and if it is necessarily unique. The question is still open.\\
Summarizing, we have several notions of maximality for a globally hyperbolic space-time. These different notions are not completely independent but there are some implications. Just as in Riemannian geometry,  constant curvature Lorentzian manifolds  are conformally flat.  Let  $M$ be  a constant curvature space-time. Then $M$ is  a $(G,X)$-manifold (where $X$ is equal to $\R^{1,n}$, $AdS_{1,n}$ or $dS_{1,n}$),  and we have:
\begin{align}\label{implicazioni massimalità}
M \ \text{is} \ C\text{-maximal}\Rightarrow M \ \text{ is }  C_0\text{-maximal} \Rightarrow M \ \ \text{is } X\text{-maximal}
\end{align}

 The converse  implications are not  true in general:  we could have a $C_0$-maximal space-time which is not $C$-maximal, or $X$-maximal space-time which is not  $C_0$-maximal, etc..
In another paper  we will develop new tools which allow us to prove that in fact  these two inverse implications are true: every $C_0$-maximal space-time  and  every $AdS_{1,n}$-maximal or $\R^{1,n}$-maximal space-time is also $C$-maximal. Conversely every $dS_{1,n}$-maximal space time always has a non trivial $C_0$-extension. 

\paragraph{Completeness of $C_0$-maximal space-times.}
In the second part of the paper we study the developing map of a maximal conformally flat space-time. We provide a sufficient and necessary condition on the causal structure of the space-time for the developing map to be a global diffeomorphism; in other words, a causal characterization of conformally flat space-times which are complete as  $(G,X)$-manifolds.
The uniqueness of the $C_0$-maximal extension in Theorem \ref{teo max}  implies:

\begin{teo}\label{teo: compatta s connessa}
The universal cover of the Einstein space-time is the only globally hyperbolic conformally flat space-time of dimension $\geq3$ which is $C_0$-maximal, simply connected, and admitting a compact Cauchy hypersurface.
\end{teo}

This result implies that a $C_0$-maximal globally hyperbolic conformally flat space-time $M$ is a finite quotient of the Einstein space-times if and only if  the lift, to its  universal cover, of every Cauchy hypersurface, is compact. Moreover, because we know very well the causal structure of the Einstein space-time, and in particular we have a very clear description of its lightlike geodesics, we have:

\begin{teo} \label{teo: conjugate}
Let $M$ be a conformally flat globally hyperbolic $C_0$-maximal space-time of dimension $\geq 3$  which has two freely homotopic lightlike geodesics which are distinct but with the same ends. Then $M$ is a finite quotient of universal cover of the Einstein space-time.
\end{teo}

In a following paper we will show some consequences of this result. It gives some information about the domains of injectivity of the developing map of a $C_0$-maximal space-time. It turns out that the developing map of a conformally flat globally hyperbolic space-time $M$ has  to be injective on the causal  past and  future of each point. Moreover, the image of these set is a regular Minkowski domain, future or past complete  (following the definition of \cite{Bonsante} and \cite{BarbotGHflat} to classify the $\R^{1,n}$-maximal globally hyperbolic  space-time with compact Cauchy hypersurface).   

\paragraph{Organization of the paper.}
In  Section \ref{sec: preliminari} we recall some classical results in causality of space-time, in particular we recall the properties of globally hyperbolic space-times. We also give a detailed description of the Einstein space-time, which will be an essential tool in the proof of the mains results.  The rigidity properties of conformal  maps and Liouville's Theorem will be also recalled in this section. In section \ref{sec: C0 extension}, after some properties of conformal Cauchy-embeddings, we give the proof of Theorem \ref{teo max}.  In Section \ref{sec: thm punti coniugati} we prove Theorem \ref{teo: compatta s connessa}  and Theorem \ref{teo: conjugate}.

\section{Preliminaries}\label{sec: preliminari}

In this section we recall some classical definitions and results about conformally flat globally hyperbolic space-times and we gives the proof of some technical lemmas and propositions which play a rule in the proof of the main results of this article. The theory of space-times has been largely studied by Hawking, Penrose and many others. A quite complete exposition of the main results can be found in  \cite{Hawking},  \cite{Beem}, \cite{ONeil}, \cite{Penrose}. For a clear exposition of the hierarchy of causal notions associated to space-times see also \cite{CausalGerarchy}.

\subsection{Causal structure of space-times}

\paragraph{Space-times.}
A  \em Lorentzian $(n+1)$-manifold \em is  a smooth $(n+1)$-manifold $M$ (which includes the topological assumption that M is metrizable and with countable basis), endowed with a symmetric non-degenerate $2$-form $g$ with signature $(n,1)$. \\
A non zero\footnote{Our convention is to consider the  zero vector as a spacelike vector. In particular, a causal
or lightlike vector is non zero.} tangent vector $v$ is \em spacelike \em (resp. \em  timelike,  lightlike, causal\em) if $g(v,v)$ is positive  (resp. \em  negative,  null, non-positive\em).\\
In each tangent space $T_pM$ the cone of causal vectors has two connected components $C_p^+$ and $C^-_p$.  The manifold $M$ is \em time-orientable \em if it is possible to make a continuous choice,  in each tangent space of $M$, of one of them.  This means that over $M$ there are two continuous cone fields: one of them is chosen to be the future one, denoted by $C^+$, and the other to be the past one, denoted by $C^-$. Such a choice defines a \em time-orientation \em on $M$. A non-spacelike vector $w\in T_pM$  is   \em future-directed  \em if  it is in $C^+$ and \em past-directed \em if it is in $C^-$.
Remark that up to a double cover $M$ is always time-orientable.

\begin{defi}
A \textbf{space-time} is a connected orientable and  time-orienta\-ble  Lorentzian manifold provided with a time-orientation.
\end{defi}

 A  \em differentiable causal curve \em (respectively \em timelike, spacelike, lightlike\em)  of  a space-time $M$  is a $C^1$ map  $c : ]a, b[ \rightarrow M$  such that at every point its tangent vector  is  causal (respectively timelike, spacelike, lightlike).
In particular, every causal curve $c$ is an immersion and the vectors $c'(t)$ have the same time orientation  for all $t \in ]a, b[ $: the curve $c$   is said to be \em future \em or \em past \em oriented following the time-orientation of  $c'(t)$.\\
The \emph{causal future} $J_U^+(A)$  (respectively \emph{chronological future} $I_U^+(A)$) of a subset $A$ of  $M$ relative to an open set $U$ is the set of future ends of all future causal (respectively timelike) curves  starting  from a point of $A$ and contained in $U$. The \em chronological past \em  and the \em causal past \em of $A$ relative to an open set $U$, noted $I_U^-(A)$ and $J_U^-(A)$, are the chronological and causal futures of $A$ for the opposite time-orientation on $M$. When $U=M$ the chronological and causal  past an future of a set $A$ are noted $I^\pm(A)$ and $ J^\pm(A)$. Since a timelike curve is in particular a causal curve we have
$I_U^\pm(A)\subset J_U^\pm(A)$. These sets give what is called the \em causal structure \em of $M$.\\
It's possible to give a more general notion of causal curve without changing the causal structure of $M$.
  A \em future  causal curve \em is a $C^0$ map $c :]a, b[\subset \R\rightarrow M$ such that for all $t_0\in ]a, b[$ and for all neighborhood
$U$ of $c(t_0)$, there exists $\varepsilon >0$ such that for all $t\in I:=]t_0 - \varepsilon , t_0 + \varepsilon [$ we have
$c(t)\in J^-_U(t_0)$ if  $t\leq t_0$ and $c(t)\in J^+_U(t_0)$ if  $t\geq t_0$. We can define the past causal curves like the future ones for the opposite chronological orientation.\\
It is not hard to prove that the causal curves just defined are more regular than continuous: they are locally Lipschitz  (see \cite{Barbot} for the proof in Minkowski space-time). Moreover, by definition, it is clear that if there is a causal curve from a point $p$ to a point $q$, then there is a piecewise differentiable causal curve from $p$ to $q$.\\
The causal structure  naturally defines two relations:  given  $x,y \in M$, we write 
$x<y$ iff $x\in I^-(y)$ and $x\leq y$ iff $x\in J^-(y)$. They are called the \em causal relations \em of $M$. By the definition it is clear that the relations $<$ and $\leq$ are transitive and that the relation $\leq$ is reflexive. If $\leq$ is also antisymmetric we say that $M$ is a \em causal space-time \em. This means that we cannot have causal closed curves in the space-time and in this case $\leq$ is a partial order on $M$.
The causal relations are more than transitive. Given $x,y,z\in M$ then
$x\leq y \text{ and } y <z $ imply $ x<z$ (and $x< y \text{ and } y \leq z$ imply $  x<z$). This is a consequence of the following proposition:

\begin{prop}[\cite{ONeil} p. 294]\label{prop: geodesiche luce e causalità}
In a Lorentz manifold   if $\alpha$ is a causal curve   from $p$ to $q$ that is not a null pregeodesic, then there is a timelike curve from $p$ to $q$ arbitrarily close to $\alpha$.
\end{prop}


\paragraph{Globally hyperbolic space-times.} \label{GH}
In Riemannian geometry it is often useful to consider an open neighborhood which is geodesically convex: this is the image by the exponential map restricted to some open neighborhood of zero. In a pseudo-Riemannian manifold we also have the exponential map and we can make a similar construction. However from the point of view of the causal structure there is another notion of convexity: an open set $U$ of a space-time \em is causally convex \em if every causal curve between two of its points is contained in $U$.\\
A natural hypothesis, if we are looking for space-times which are interesting from a physical point of view, is to require that the space-time contains no closed causal loop (physically   time-travel  is not allowed).  For example, this is the case for causal space-times.  However often this  is not enough to be physically useful;  for instance a curve should  not return arbitrary near to its starting point. This is precisely what happens in \em strongly causal \em space-times:   a space-time   is  strongly causal  if every point has a causally convex neighborhood.

\begin{defi}\label{def: GH diamante compatto}
A space-time $M$ is  \emph{globally hyperbolic} if it is  strongly causal and for every  $x, y \in M$ such that  $y\in J^-(x)$,  the intersection $J^+(y) \cap J^-(x)$ is compact.
\end{defi}

 This is the classical definition of globally hyperbolicity, however Sanchez has recently proved, in \cite{Sanchez}, that the hypothesis "$M$ is strongly causal" can be replaced by the hypothesis "$M$ is causal". \\
One of the main properties of globally hyperbolic space-times is the following \footnote{In fact, as explained in \cite{Penrose}, the definition of limit curve has been  fitted in order to have this convergence property for globally hyperbolic space-times.}:

\begin{lem}[ \cite{Beem} p. $78$ (corollary $3.32$) and \cite{ONeil} p. $405$ (Proposition $8$)]\label{lemma: curva limite causale}
Let $M$ be a globally hyperbolic space-time. Let $\{p_n\}_{n\in \N}$ and $\{q_n\}_{n\in\N}$ be two sequences of points in $M$ such that $p_n\rightarrow p$ and $q_n\leq p_n$ and let $\gamma_n$ be, for every $n$, a past causal curve from $p_n$ to $q_n$. Then:  
\begin{itemize}
\item if $\exists q\neq p$  such that $q_n\rightarrow q$ and $q\leq p $, then the $\gamma_n$ have a  limit curve going from $p$ to $q$ and which is past and causal,
\item if the sequence $\{q_n\}_{n\in\N}$ is unbounded, then there exists a past inextensible causal curve starting from $p$ which is a limit curve for $\gamma_n$.
\end{itemize}
\end{lem}

 An \em achronal \em (\em acausal \em) subset of a space-time is a subset which intersects every timelike (causal) curve in at most one point.

\begin{defi}\label{def: sviluppo di Cauchy}
\emph{Let $A$ be an achronal subset of  $M$. The \emph{future  Cauchy development} (resp. \emph{past}) of $A$, written $\mathscr{D}^+(A$) (resp. $\mathscr{D}^-(A)$), is the set of points  $x$ of $M$ in the chronological future (resp. past) of $A$ such that every past  (future) inextensible causal curve starting from $x$ intersects $A$.\\
The  \em Cauchy development \em of $A$ is the union  $$\mathscr{D}(A):=A \cup \mathscr{D}^+(A)\cup \mathscr{D}^-(A).$$}
\end{defi}

\begin{prop}[\cite{ONeil} p. 421]\label{prop: intD(A) GH}
Let $A$ be an achronal subset of $M$. If  $int(\mathscr{D}(A))$ is not empty, then it is globally hyperbolic.
\end{prop}


\begin{defi}\label{def: edgeless}
\emph{A locally achronal subset $A$ of $M$ is said to be \emph{edgeless} if
for every $x$ of $A$ there exists a neighborhood $U$ of $x$ such that:
\begin{itemize}
\item  $U \cap A$ is achronal relative to $U$: every timelike curve contained in $U$ intersects $U\cap A$ at most in one point,
\item  every causal curve contained in $U$, which starts  from  a point of $I_U^-(x)$
and ends in  $I_U^+(x)$, must intersect  $U \cap A$.
\end{itemize}}
\end{defi}

\noindent A subset $V$ of a space-time $M$ is  a \em past set \em(reps. \em future set \em)  if  $I^-(V)\subset V$ (resp. $I^+(V)\subset V$).

\begin{lem}[\cite{ONeil} p. 414 Corollary 26 and 27]\label{lemma: frontiera di insieme passato}
The (non empty) boundary of a past (future) set $P$ is a closed achronal  and edgeless topological hypersurface.
\end{lem}

 In general every locally achronal edgeless subset of a space-time $M$ is an embedded topological hypersurface (see  Lemma 1.2.28  of \cite{Barbot} for  Minkowski space-time).




\begin{lem}\label{lemma: developpement vide}
Let $A$ be an achronal edgeless subset of a strongly causal space-time $M$. Then if $\mathscr{D}^+(A)=\emptyset$ ($\mathscr{D}^-(A)=\emptyset$), for every point $p\in A$ there exists a past (future) inextensible lightlike  geodesic
$c:[0,\infty[\to M$ starting from $p$  such that $c\cap A$ is a  past  (future) lightlike   geodesic   contained in  $A$ without past (future) limit point in $A$.
\end{lem}

\begin{proof}
This lemma  results from the theory of the \em Cauchy horizon \em for the Cauchy development of achronal subsets. 
A proof can be found in \cite{Hawking} p. 203 Proposition 6.5.3 and its corollary.
\end{proof}

\begin{defi}\label{hypersurface de Cauchy}
 A \emph{Cauchy hypersurface} of $M$ is a closed achronal edgeless set $S\subset M$  such that $\mathscr{D}(S)=M$.\\
 A \emph{Cauchy time-function}  is a continuous map $t: M\rightarrow \mathbb{R}$ such that, for every inextensible future causal curve $c$ of $M$, $t\circ c$ is increasing and onto.  In particular every level set of $t$ is a Cauchy hypersurface for $M$. 
\end{defi}

 By Proposition \ref{prop: intD(A) GH}, if $M$ has a Cauchy hypersurface then it is globally hyperbolic. The converse is a consequence of the following more general result, called Geroch's Theorem.

\begin{teo}\label{teo: Geroch}
A space-time $M$ is globally hyperbolic if and only if it admits a $C^\infty$ Cauchy time-function.
\end{teo}

 This result, originally proved in  \cite{GerochTimeFunction}, has been rewritten several times, in order to correct some mistakes in the proof of the regularity of the Cauchy time function. Another reference for the proof of the existence of continuous Cauchy time function is  \cite[p. 1155]{SachsWuTimeFunction1977}.

\begin{cor}\label{cor: Geroch produit}
Every globally hyperbolic space-time is homeomorphic to a product $S\times \mathbb{R}$. Furthemore, for every $t\in \R$, the projection on to the  factor  $\mathbb{R}$  is a Cauchy time-function.
\end{cor}

In general it is not  true that a  closed achronal edgeless subset  $A$ of a globally hyperbolic space-time  is a Cauchy  hypersurface\footnote{ The set $A=\{x\in \R^{1,n} \ :\ \Vert x\Vert_{1,n}=-1\}$ is a closed achronal edgeless subset of $\R^{1,n}$. But it is not a Cauchy hypersurface, because no lightlike straight lines going through the origin intersect $A$. }. However this is true under a compactness hypothesis:

\begin{prop}\label{prop: ipersuperficie de Cauchy compatta}
Let $S$ be an achronal edgeless  compact subset of a
globally hyperbolic space-time  $M$. Then $S$ is a Cauchy hypersurface of $M$.
\end{prop}

\begin{proof}
By hypothesis $S$ is  closed achronal edgeless so we just have to verify that $S$ intersects every inextensible causal curve of $M$.\\
We start by proving that $\partial I^-(S)=\partial I^+(S)=S$. Assume by contradiction that there exists $p\in \partial I^-(S)\setminus S$. Let $\{p_n\}_{n\in \N}$ a sequence in $I^-(S)$ converging to $p$. For every $n$ there exists a point $z_n$ in $S$ such that $p_n\in I^-(z_n)$. Since $S$ is compact, up to a subsequence, we can assume that the sequence  $\{z_n\}_{n\in \N}$ converges to a point $z$. Then Lemma \ref{lemma: curva limite causale} applies: there exists a past causal curve $\gamma$ from $z$ to $p$. Since $p\in \partial I^-(S)\setminus S$; the curve $\gamma$ has to be a lightlike geodesic (according to Proposition \ref{prop: geodesiche luce e causalità}). This implies that every point  $p'\neq p$ which is in $\gamma$ is also in $\partial I^-(S)\setminus S$.\\
 Let $c$ be a timelike curve from a point  $p'$ in $\gamma$ to a point $z'\in I^+(z)$. Since $S$ is edgeless,  if $p'$ and $z'$ are sufficiently near  $z$, the curve $c$ has to meet $S$ in a point $q$. By Proposition \ref{prop: geodesiche luce e causalità} we have  $p\in I^-(q)$, and then $p\in I^-(S)$, which is a contradiction.\\
A similar argument can be used to prove $\partial I^+(S)=S$, so we have the disjoint union 
$M=I^+(S)\sqcup S\sqcup I^-(S)$. \\
Let $\tau:M\rightarrow \R$ be a  Cauchy time function. Since $S$ is compact $\tau\vert _S$ has a maximum $A$. We define $\Sigma:=\tau^{-1}(b)$ where $b>A$: then $\Sigma$ is a Cauchy hypersurface contained in $I^+(S)$.\\
Let $p$ be a point in $I^-(S)$ and let $\alpha$ be a future inextensible causal curve starting from $p$. The curve $\alpha$ intersects $\Sigma$. Since $\Sigma$ is strictly contained in $I^+(S)$, $\alpha$ must intersect the boundary of $I^-(S)$  before intersecting $\Sigma$. Since the boundary of  $I^+(S)$ is $S$, $\alpha$ intersects $S$.\\
In the same way we can show that every past inextensible causal curve starting from of $I^+(S)$ must intersect  $S$.
This shows that  $\mathcal{D}(S)=M$ and thus $S$ is a Cauchy hypersurface.
\end{proof}

\subsection{Conformally flat space-times}\label{sezione: Ein}
\paragraph*{Conformal maps and causality.}
There is a natural question to ask: when do two different Lorentzian metrics  define the same causal structure on a manifold? A sufficient condition is that the two metrics be in the same conformal class. Indeed, if in each tangent space we multiply  the  metric by a positive constant, the causal type of the tangent vectors does not change, and so the causal structure of the entire manifold is preserved\footnote{For intellectual satisfaction we should also mention that this condition is also necessary for strongly causal space-times. More precisely: in \cite{Hawking_applicazioniConformi} S.W. Hawking, A.R. King and P.J.W. McCarthy have shown that if $g$ and $g'$ define the same causal structure over $M$ and if this structure is strongly causal, then $g$ and $g'$ are conformally equivalent. The hard part of their proof is to show that every homeomorphism which preserves the causal structure of a strongly causal space-time is differentiable. Starting from that it is easy to show that the identity is a conformal map}. 
Then every result about the causal structure of a given space-time is true for all the Lorentzian metrics in the same conformal class. However, in general, it is clear that two different  metrics in the same conformal class have different  geodesics. This is because the  Levi-Civita connexion is not preserved by conformal changes. The formula which gives the new Levi-Civita connexion, after a conformal change of metric, can be found in \cite{Beem}, chapter $9$. By this formula it is not hard to prove:

\begin{lem}[\cite{Beem}, Proposition $9.17$]\label{g\'eod\'esiques lumieres}
Let $(M,g)$ a space-time, and $f: (M,g)\rightarrow (M,g)$ a conformal map. Then the image  by $f$  of every lightlike geodesic of $M$, up to parametrization, is a lightlike geodesic  of $N$
\end{lem}

\noindent A different and nice proof of Lemma \ref{g\'eod\'esiques lumieres}, using the fact that lightlike geodesics are the solutions in the zero level of a Hamiltonian system, can be found in \cite{Frances} (chapter 1, p. 14).

\paragraph*{Einstein space-time.}
It is well known that the conformal sphere can be identified to the boundary of the hyperbolic space of higher dimension. This construction has a Lorentzian analog: the Einstein space-time can be identified to the conformal boundary of the anti-de Sitter space-time.\\
Let  $\R^{2,n+1}$ be the vector space $\R^{n+3}$  with the canonical quadratic form of signature $(2,n+1)$ and let $C$ be the cone of isotropic vectors. Let $S(\R^{2,n+1})$  the quotient of $\R^{2,n+1}$ by positive rescaling, and let   $\pi: \R^{2,n+1}\to S(\R^{2,n+1})$ be the associated projection. Since the space-time $Ad\Sf_{1,n}$ is defined as the set of vectors of $\R^{2,n+1}$ which have norm $-1$, the map $\pi$ is injective on this set. The space-time $Ad\Sf_{1,n}$ is then identified to its image, and we call it the projective model of $Ad\Sf_{1,n}$. The boundary of the image of $Ad\Sf_{1,n}$ by $\pi$ is the image of $C$.  \\
 For every $v\in \R^{2,n+1}$ the kernel of $d_v(\pi\vert_C)$ is exactly the degenerate direction  of the ambient quadratic form of  $\R^{2,n+1} $ restricted to $T_pC$.  Given two sections $\phi, \phi': \pi(C)\to C$, there is always a positive function $f$ defined on $\pi(C)$ such that for every $x\in \pi(C)$, we have $\phi(x)=f(x)\phi'(x)$. The two metrics defined on $\pi(C)$ by the pull backs  by $\phi$ and $\phi'$  of the quadratic form of $\R^{2,n+1}$, are then conformally equivalent, with conformal factor $f^2$.  In other words, the quadratic form of $\R^{2,n+1}$ naturally defines a conformal class of Lorentzian metrics $[g]$ over $\pi(C)$.

\begin{defi}
 The \em Einstein space-time \em  of dimension $n+1$, denoted $Ein_{1,n}$, is  the topological space $\pi(C)$ endowed with the conformal class  of Lorentzian metrics $[g]$.
 \end{defi}
 It turns out that the  conformal class $[g]$ of $Ein_{1,n}$ is conformally flat:  that is, every point $p$ of  $Ein_{1,n}$ has a neighborhood $U$ such that  $[g] \vert_U$  contains a flat metric.  This fact is not evident \em a priori, \em  it comes from the fact that  the model flat space-time, the Minkowski space-time $\R^{1,n}$,  admit conformal embeddings  in $Ein_{1,n}$. Since the action of the group  $O(2,n+1)$ on  $Ein_{1,n}$ is transitive, every point of the space-times  $Ein_{1,n}$,  has a neighborhood  conformally equivalent to $\mathbb{R}^{1,n}$. \\ 
Moreover,  the other two models of constant curvature space-times, $d\Sf_{1,n}$  and $ Ad\Sf_{1,n}$, also conformally embeds  in $Ein_{1,n}$, then we obtain that every constant curvature space-times is conformally flat  (see \cite{BarbotQuentin} Proposition 2.3 for the $Ad\Sf$ case, see \cite{PenroseInfinity} and section 2.3 of \cite{FrancesEssential} for the others). This  situation is similar to the Riemannian case:  the euclidean and the hyperbolic space conformally embed into the Riemannian sphere, where the orthogonal group acts transitively.\\
 Is not hard to see that the space-time $Ein_{1,n}$ can be identified with the product $\mathbb{S}^{n}\times \mathbb{S}^1$, equipped with the conformal class of the metric $d\sigma^2-d\theta^2$, where $d\sigma^2$ and $d\theta^2$ are the canonical metrics over  $\mathbb{S}^{n}$ and $\mathbb{S}^1$. This can be seen by looking at the intersection between $C$ and the sphere of radius $2$ for the canonical euclidean metric of $\R^{n+3}$. By this identification $Ein_{1,n}$ is clearly an orientable and time orientable manifold and thus a space-time.\\
The orthogonal group $O(2,n+1)$ acts transitively and faithfully on $C$ and this action preserves straight lines. Hence $O(2,n+1)$ acts transitively and faithfully over $Ein_{1,n}$ and preserves its conformal class of metrics.\\

\begin{lem}\label{lemma: curve causali Ein}
Every causal (timelike) curve $c$ of $Ein_{1,n}$ can be parameterized as $(x(t),e^{i2\pi t})$, where $x(t)$ is a (strictly) $1$-Lipschitz map from an interval of $\R$ into $\Sf^n$.
The lightlike geodesics of  $Ein_{1,n}$ are the causal curves $c$ such that, in the previous parametrization,   $x(t)$ is a geodesic of $\Sf^n$ parameterized by its arc length.
\end{lem}

\begin{proof}
Let $s\in I\subset\R\longmapsto c(s)=(w(s),p(s))\in \Sf^n\times \Sf^1$ a future causal curve in  $Ein_{1,n}$. First suppose that $c$ is $C^1$ piecewise. Then $c$ must verify:
\begin{eqnarray}\label{ineg 1}
\Vert w'(s)\Vert^2\leq \vert p'(s)\vert^2
\end{eqnarray}
This implies that the vector $p'(s)$, tangent to $\Sf^1$, never vanishes. The application $s\longmapsto p(s)$ can then be written as $p(s)=e^{i\phi(s)}$, where $s\longmapsto \phi(s)$ is a monotone map from $I$ to  an interval $J$ of $\R$. Since $c$ is a future causal curve the map $\phi(s)$ is a strictly increasing map. Then, changing the parameter  $s$ into a parameter $t:=\phi(s)$, we have $c(t)=(x(t), e^{it})$, for all $t$ in   $J$, where $x:=w\circ \phi^{-1}$, and  $\Vert x'(t)\Vert^2\leq 1$.  If we integrate the formula between two points $t$ and $t'$ of $J$, we have:
\begin{eqnarray}\label{ineg 2}
d_0(x(t),x(t'))\leq \vert t-t'\vert
\end{eqnarray}
where $d_0$ is the distance over $\Sf^n$ for the canonical metric. The map $t\in I\longmapsto x(t)$ is $1$-Lipschitz; moreover $c$ is timelike if and only if the inequality   (\ref{ineg 1}) is strict, and  (\ref{ineg 1}) is strict  if and only if  the inequality in (\ref{ineg 2}) is strict. Therefore, $C^1$ timelike curves are strictly $1$-Lipschitz. The curve $c$ is a lightlike geodesic if and only if the inequality  (\ref{ineg 1}) is an equality, and this is true if and only if (\ref{ineg 1}) is an equality, that is, if  $x$ is a geodesic of $\Sf^n$. Hence, the lemma is proved for $C^1$ curves.\\
Now assume that $s\in I\longmapsto c(s)=(w(s),p(s))$ is a topological  causal curve (not necessary $C^1$). By definition, given $s<s'$ close one to the other, there exists a non trivial $C^1$ causal curve between $c(s)$ and $c(s')$. Since the result is proved in the case of $C^1$ curves,  we have:
$$0<d_0(w(s),w(s'))\leq \vert p(s)-p(s')\vert$$
Therefore we can write  $p(s)=e^{i\psi(s)}$, where $ \psi(s)$ is strictly increasing map from $\R$ to an interval $K$ of $\R$,  and then  $c(t)=(x(t), e^{it})$, where   $x:=w\circ \psi^{-1}$ satisfies the  inequality  (\ref{ineg 2}), that is, $x$ is  $1$-Lipschitz. As before, we can see that $c$ is timelike if and only if $x$ is strictly $1$-Lipschitz, and that $c$ is a lightlike geodesic if and only if $x$ is a geodesic of $\Sf^n$.
 The lemma is proved.
\end{proof}

\begin{cor}\label{cor: Ein totally vicious}
$Ein_{1,n}$ is totally vicious, i.e. the past and the future of every point is the entire space-time.
\end{cor}

 Since $Ein_{1,n}$ is totally vicious, its causal structure gives no information: every point is causally related to any other point. However its universal covering is globally hyperbolic and has a well understood causal structure.  \\ The universal covering of the Einstein space-time,  $\E_{1,n}$, is identified to $(\Sf^n\times\R,[d\sigma^2-dt^2])$, where $dt^2$ is the canonical metric over $\R$.
 Let $pr:\E_{1,n}\rightarrow Ein_{1,n}$ be the covering map.
The fundamental group of $Ein_{1,n}$ is isomorphic to $\mathbb{Z}$: it can be identified with the cyclic group generated by the map $\delta :\E_{1,n} \rightarrow \E_{1,n}$ which  associates to $(x,t)$ the point $(x,t+2\pi)$. This is clearly a conformal diffeomorphism of $\E_{1,n}$.\\
The antipodal map of $\R^{2,n+1}$, $x\in \R^{2,n+1} \longmapsto -x$, defines a map $\overline{\sigma}: Ein_{1,n}\rightarrow Ein_{1,n} $, which is the product of the two antipodal maps of $\Sf^n$ and $\Sf^1$. The map $\overline{\sigma}$ lifts to  $\E_{1,n}$ giving the map $\sigma : \E_{1,n}\rightarrow \E_{1,n}$ which associates to $(x,t)\in \Sf^n\times\R $ the point $(-x, t+\pi)$. Then $\sigma^2=\delta$.\\

\begin{defi}\label{def: punti conniugati in E}
\emph{Two points $x$ and $y$ of  $\E_{1,n}$ are \em conjugate \em if one is the image by $\sigma$ of the other.}
\end{defi}

Since the projection $pr: \E_{1,n}\to Ein_{1,n}$ is a conformal map, given a  causal curve $c:$ in $\E_{1,n}$, the curve $pr\circ c$ is a causal curve of $Ein_{1,n}$. Then by Lemma \ref{lemma: curve causali Ein} we have:

\begin{lem}\label{lem: curve causali E}
Every causal curve $c$ of $\E_{1,n}$ can be parameterized as $c(t)=(x(t),t)$, where $x(t)$ is a $1$-Lipschitz map from an interval of $\R$ into $\Sf^n$.
The lightlike geodesics of  $\E_{1,n}$ are the causal curves $c$ such that, for the previous parametrization,   $x(t)$ is a geodesic of $\Sf^n$ parameterized by its arc length (see Figure \ref{fig:geoloceEin}).
\end{lem}

A causal curve of $\E_{1,n}$ is inextensible  if the   parametrization given by the previous lemma is defined for every $t$ in $\R$. It is then easy to see that $\E_{1,n}$ is a globally hyperbolic space-time,  with Cauchy hypersurfaces homeomorphic to $\Sf^n$: the map  $(x,t)\in \Sf^n\times \R\simeq \E_{1,n}\to \R$  is a Cauchy time function.

\begin{figure}
\centering 
\includegraphics[height=6 cm]{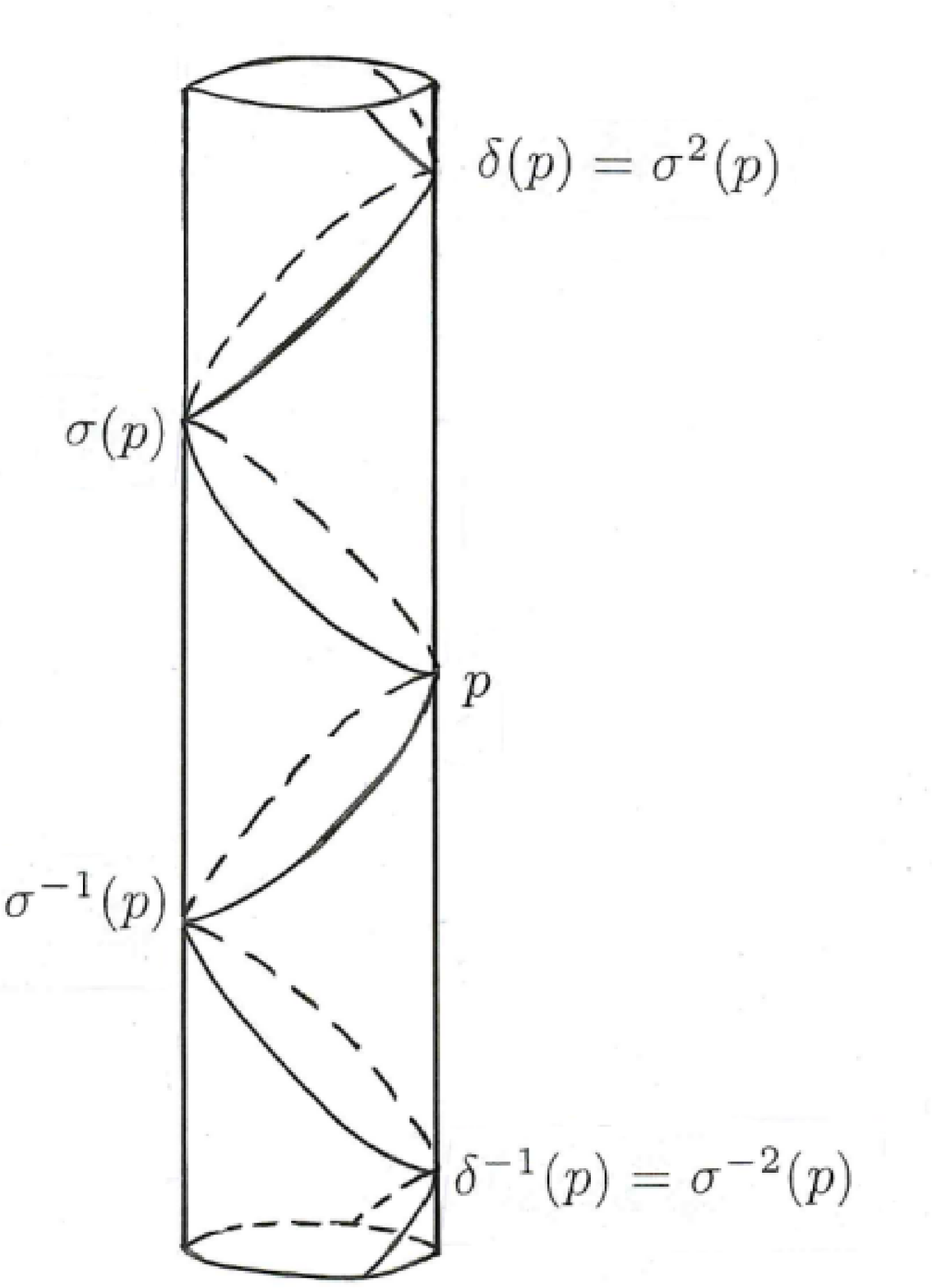}
\caption{ \protect Two inextensibles lightlike geodesics in  $\E_{1,n}\simeq \R\times \Sf^n$ \label{fig:geoloceEin}}
\end{figure}
 
Thanks to Lemma \ref{lem: curve causali E} we can understand the causal structure of  $\E_{1,n}$:

\begin{lem}\label{lemma: passato e futuro in E}
Let $p=(x,t)\in \E_{1,n}$. Then
\begin{eqnarray*}
I^\pm(p) & = & \{ (x',t')\in \E_{1,n} / d_0(x,x')< \pm(t'-t) \} \\
J^\pm(p) & = & \{ (x',t')\in \E_{1,n} / d_0(x,x')\leq \pm(t'-t) \}
\end{eqnarray*}
where $d_0$ is the canonical distance on $\Sf^n$.
\end{lem} \qed

\begin{oss}\label{oss: geodesiche luce E intersezioni}
The inextensible lightlike geodesics  starting from a point $p$ of $\E_ {1,n}$ have common intersections at all the points $\sigma^k(p)$, for $k\in \mathbb{Z}$. Outside these points, they are pairwise disjoint.
\end{oss}


\paragraph{Rigidity of conformal maps.}
Liouville's Theorem is originally  a theorem of conformal Riemannian geometry  stating that, in dimension $n\geq 3$, every conformal map between two open sets of the sphere is the restriction of an unique element of $O^+(1,n)$. This implies that the group of conformal transformations of the sphere $ \Sf^n$ (also called the group of  M\"obius transformations)  is exactly $O^+(1,n)$. \\
 This theorem has been generalized by C. Frances (see \cite{Frances} and \cite{Cahenkerbrat82}) to pseudo-Riemannian conformally flat metrics. This has been possible  because the Liouville's Theorem is an aspect of a more general phenomena: the rigidity of conformal maps, between pseudo-Riemannian manifolds of dimension greater or equal than $3$.\\ Let    $M$ and $N$ be two manifolds. Denote by $ \mathrm{Diff}_{loc}(M,N)$ the set of local diffeomorphism  between  $M$ and $N$. On $ \mathrm{Diff}_{loc}(M,N)$ we have the following equivalence relation: two local diffeomorphisms $f,g: M \to N$ are said to be equivalent if in some local chart they have the same Taylor polynomial up to order $r$ at $x$. The  \em $r$-jet \em of $f\in \mathrm{Diff}_{loc}(M,N) $ at a point $x\in M$, denoted by $j_x^rf$, is the equivalence class of $f$ for this relation. We have the following rigidity result:

\begin{teo}\label{rigidite}
Let $M$ and $N$ be two pseudo-Riemannian manifolds of dimension $\geq 3$. Let $f$ and $g$ be two conformal maps from $M$ to $N$. If  $f$ and $g$ have the same $2$-jet at one point of $M$, then they are equal.
\end{teo}

\noindent A proof of this result can be found in \cite{Frances} (chapter $2$). For  more details about rigidity of conformal application see also \cite{Kobayashi}  and \cite{Ruh}. \\
Liouville's Theorem is then a consequence of Theorem \ref{rigidite}.

\begin{teo}(Liouville)\label{Liouville}
For every $n\geq 2$, every conformal map between two open sets of  $Ein_{1,n}$ is the restriction of a unique element of   $SO(2,n+1)$.
\end{teo}

In particular this implies that the conformal group of $Ein_{1,n}$ is exactly $O(2,n+1)$ and the group of conformal maps which preserve the orientation and the time-orientation is the connected component of the identity, that we denote by $O_0(2,n+1)$. \\
Every conformal diffeomorphism of $Ein_{1,n}$ lifts to a conformal diffeomorphism of   $\E_{1,n}$. By Liouville's Theorem,  when  $n\geq 2$, the reverse statement is also true: every conformal diffeomorphism of $\E_{1,n}$ defines a unique conformal diffeomorphism of $Ein_{1,n}$.  So we have a surjective morphism $j: \operatorname{Conf}(\E_{1,n})\rightarrow \operatorname{Conf}(Ein_{1,n})$. The kernel of $j$ is the subgroup generated by $\delta$; it is contained in the center of $\operatorname{Conf}(\E_{1,n})$. Then $\operatorname{Conf}(\E_{1,n}) = \widetilde{O(2,n+1)}=O(2,n+1)\rtimes \Z$.

\begin{cor}
Every conformally flat space-time of dimension greater then $3$ is locally modeled on $(Ein_{1,n+1},O_0(2,n+1))$.
\end{cor}

\noindent The proof of this result is quite standard, it can be found for example in \cite{Matsumoto} in the   case of Riemannian conformal geometry. In fact in Riemannian geometry we have the same situation:  Liouville's Theorem implies that every conformally flat Remannnian manifold is locally modeled on the conformal sphere endowed with the action of its group of conformal transformations (see \cite{Matsumoto}). More details on the general theory of $(X,G)$-manifolds can be found in  \cite{Goldman}.

\section{$C_0$-maximum extension}\label{sec: C0 extension}
 
\subsection{Cauchy-embeddings.}
In this section $M$ and $N$ will always assumed to be globally hyperbolic space-times with Cauchy hypersurfaces $S$ and $S'$ respectively. We say that a conformal  map $f: M \to N$ is a \textit{(conformal) Cauchy-embedding} if  $S'=f(S)$;  $f$ is also denoted  by $f: (M,S) \rightarrow (N,S')$. If moreover $f$ is an isometry, we say that $f$ is an \textit{isometric Cauchy-embedding}.\\
To prove the existence and uniqueness of the conformally flat maximal extension we need some technical results about conformal Cauchy-embeddings. Since these results only involve causal properties, it makes no difference whether we consider conformal or isometric Cauchy-embeddings.

\begin{lem}\label{lemma: l'immagine Cauchy-plong. e causalemente convessa}
The image of any conformal Cauchy-embedding $f: (M,S)\rightarrow (N,S')$ is a causally convex open subset of $N$.
\end{lem}

\begin{proof}
Let  $x$ and $y$ be two points of  $M$ such that there exists  a future causal curve   $\alpha:[0,1]\rightarrow N$ with $c(0)=f(x)$ and $c(1)=f(y)$. Let  $c:\mathbb{R}\rightarrow N$ be an inextensible future causal curve which extends $\alpha$. Since $N$ is globally hyperbolic,  $c:\mathbb{R}\rightarrow N$ is an embedding. Then, since $f(M)$ is an open subset of $N$, the intersection $c^{-1}(f(M)\cap c(\mathbb{R}))$ is the union of disjoint segments (the connected components). Let $I$ be one of them. The curve $C:=(f\vert_{f(M)})^{-1}\circ c: I\rightarrow M$  is a causal inextensible curve in $M$; therefore, $C(t)$ intersects the Cauchy hypersurface   $S$ of $M$ at a point $z\in M$. It implies $f(z)\in f(S)\cap c(I)\neq  \emptyset $. Since $f(S)$ is a  Cauchy hypersurface it meets every causal curve at most at one point. Then  $c^{-1}(f(M)\cap c(\mathbb{R}))$ has only one connected component which is the entire interval $[0,1]$. It follows that $f(M)$ is causally convex.
\end{proof}

\begin{cor}\label{cor: immagine di acronale per un Cauchy-plongement}
The image of every achronal set $A$ of $M$ by any conformal Cauchy-embedding $f: (M,S) \rightarrow (N,S')$ is an achronal  subset of $N$.
\end{cor}

\begin{proof}
Assume by contradiction that there is a timelike curve $\gamma$ between two points of $f(A)$. Since $f(M)$ is causally convex  in $N$, the curve $\gamma$ is completely contained in $f(M)$. Then the curve $(f\vert_{f(M)})^{-1}\circ \gamma$ is a timelike curve between two points of $A$: this contradicts the hypothesis.
\end{proof}

\begin{cor} 
A conformal Cauchy-embedding $f: (M,S) \rightarrow (N,S')$ sends every Cauchy hypersurface of $M$ to a Cauchy hypersurface of $N$.
\end{cor}

\begin{proof}
Let $\Sigma\subset M$ a Cauchy hypersurface. According to Corollary \ref{cor: immagine di acronale per un Cauchy-plongement}, $f(\Sigma)$ is an achronal  hypersurface. Since $f(M)$ is an open neighborhood of $f(\Sigma)$ and  $\Sigma$ is edgeless,  $f(\Sigma)$ is edgeless too. We have to show that $f(\Sigma)$ intersects every inextensible causal curve. Let $c:\mathbb{R}\rightarrow N$ be an inextensible causal curve of  $N$. We know that $c$ intersects $S'$, and this implies $f(M)\cap c(\mathbb{R})\neq \emptyset $. By the previous proof   $(f\vert_{f(M)})^{-1}\circ c: \mathbb{R}\rightarrow M$ is a inextensible causal curve of $M$.  Since $\Sigma$ is a Cauchy hypersurface it intersects this curve. Hence  $c$ intersects $f(\Sigma)$: it shows that $f(\Sigma)$ is a Cauchy hypersurface  of $M'$.
\end{proof}

\begin{lem}\label{primo}
Let $U$ be an open neighborhood of the Cauchy hypersurface $S$ in $M$.
Let $f,g : M \rightarrow N$ be two Cauchy-embeddings such that  $f\vert_U = g\vert_U$. Then $f=g$.
\end{lem}

\begin{proof}
The set $\mathcal{U}:=\{x\in M \ :\ j_x^2f=j^2_xg\}$ is a closed subset of $M$, which is non-empty since $f$ and $g$ coincide in $U$. By Theorem \ref{rigidite}, $\mathcal{U}$ is also open.  Since $M$ is connected, we have  $\mathcal{U}=M$. 
\end{proof}

\begin{lem}\label{lemma: frontiera Cauchy-plongement}
Let  $f: M\rightarrow N$ be a Cauchy-embedding. The boundary $\partial f(M)$ is the union of two disjoint closed achronal edgeless  subsets  (each possibly empty) $\partial^+ f(M)$ and $\partial^- f(M)$ of $N$ such that
$$ I^-(\partial^+ f(M))\cap I^+(\partial^- f(M))\subset f(M).$$
\end{lem}

\begin{proof}
Let $S$ be a Cauchy spacelike hypersurface of $M$. We identify  $M$  with its image in $N$ by $f$; in particular, we consider $S$ as a Cauchy hypersurface of $N$. Let $N^\pm:=I^\pm(S)\cap N$ (where $I^\pm(S)$ denote the future/past of $S$ in $N$) and let $M^\pm:=N^\pm \cap M$.
The boundary $\partial M$ is then the disjoint union of $\partial^+M:=\partial M\cap N^+$  and $\partial^-M:=\partial M\cap N^-$. \\
\em $1)$ For every point $p\in \partial^+M$ we have $I^-(p)\cap N^+\subset M^+$. \em \\
Let  $q\in I^-(p)\cap N^+$. There exists a past causal curve $c$ between $q$ and a point $z\in S$. For every $w$ sufficiently close to $p$ we have  $q\in I^-(w)$. Since $p$ lies in the boundary of $M^+$, we can select such a $w$ in $M^+$. Then, there exists a past causal curve going from  $w$  to $z$ through $q$. Since $M$ is causally convex in $N$ (Lemma \ref{lemma: l'immagine Cauchy-plong. e causalemente convessa}) we obtain  $q\in M^+$.\\
Reversing the time orientation we have also proved:\\
\em $1')$ For every point $p\in \partial^-M$ : $I^+(p)\cap N^-\subset M^-$. \em   In particular,
$$ I^-(\partial^+ f(M))\cap I^+(\partial^- f(M))\subset f(M).$$
\em $2)$  $I^+(\partial^+M) \cap M$  is empty. \em \\
Assume by contradiction that there is an element $x$ of  $I^+(\partial^+M) \cap M$. There exists a past causal curve $c$ between  $x$ and a point  $y$ in $ \partial^+M$. Extend $c$ to an inextensible (in $N$) past causal curve $c'$. Then $c'$ intersects $S$ at a point  $z$. Since $M$ is causally convex in  $N$, consequently $y\in M$, contradicting the hypothesis.\\
\em $3)$ $\partial^+M$ is achronal. \em \\
Let $c$ be a timelike future curve linking two points $x<y$ of $\partial^+M$.  Then $I^+(x)$ is an open neighborhood of $y$, and by $2)$ it is disjoint from $M$: but this contradicts the fact that $y$ lies in the closure of $M$.\\
\em $4)$ $\partial^+M$ is edgeless. \em \\
  Every causal curve between  a point in $N^+\setminus M^+$ and a point in $M^+$ intersects $\partial^+M$. By $2)$ we have that $\partial^+M$ is edgeless.\\
Reversing the time orientation we have the same results for  $\partial^- M$.
\end{proof}

\subsection{Existence and uniqueness of the $C_0$-maximum extension}

Let $M$ be a globally hyperbolic space-time of dimension $n+1\geq 3$, and let $\phi: \Sigma \rightarrow M$ be a conformal embedding of a Riemannian manifold $\Sigma$ in $M$ such that $\phi(\Sigma)$ is a Cauchy hypersurface of $M$.
Let $\mathcal{F}$ be the set of triples $(N,\psi,f)$, where:

\begin{itemize} 
\item  $N$ is a globally hyperbolic space-time,
\item $\psi:  \Sigma \rightarrow  N$ is a conformal embedding such that $\psi(\Sigma)$ is a Cauchy hypersurface of $N$,
\item $f: M\rightarrow N$ is a conformal Cauchy-embedding such that $f\circ \phi=\psi$.
\end{itemize}

\noindent We can define the following relation over $\mathcal{F}$:
  \begin{eqnarray*}
 (N,\psi,f)\preceq  (N',\psi',f') &  \iff  \ \exists \  h: N \rightarrow N'  \text{ conformal embedding}  \\
  &\text{such that } h\circ\psi=\psi'. 
\end{eqnarray*}

\noindent The fact that $h\circ\psi=\psi'$ implies  $h\circ f=f'$. Moreover, by Lemma \ref{primo}, if $(N,\psi,f)\preceq (N',\psi',f')$,   then the Cauchy-embedding $h: N \rightarrow N'$ such that $h\circ\psi=\psi'$ is unique.\\
The relation $\preceq$ is clearly reflexive and transitive, but not antisymmetric. Nevertheless,  if $(N,\psi, f)\preceq (N',\psi', f')$ and $(N',\psi', f')\preceq (N,\psi, f)$ then there are two Cauchy-embeddings  $h: N \to N'$ and $h': N' \to N$ such that $h\circ\psi=\psi'$ and $h'\circ\psi'=\psi$. The restriction of $h\circ h'$ to  the Cauchy hypersurface  $\psi(\Sigma)$ is the identity, hence, by Lemma \ref{primo}, $h \circ h' $ is the identity map on $N$.
  Similarly, $h' \circ h$ is the identity map of $N'$. We have proved that $N$ and $N'$ are conformally diffeomorphic.\\ In order to obtain a partial ordered set we consider over $\mathcal{F}$ the relation:
\begin{eqnarray*}
(N,\psi,f)\simeq (N',\psi',f') & \iff  (N,\psi,f)\preceq (N',\psi',f') \text{ and }\\ 
&(N',\psi',f')\preceq (N,\psi,f)
\end{eqnarray*}
This is an equivalence relation on $\mathcal{F}$.  For every element $(N,\psi,f)$ of $\mathcal F$, we denote by $[N,\psi,f]$ the equivalence class of $(N,\psi,f)$. Let $\overline{\mathcal F}$ be the quotient set $\mathcal F/ \simeq$. The relation $\preceq$ induces a partial order on $\overline{\mathcal F}$. Observe
  that $[M, \phi, Id]$ is a minimum, i.e. it minorates every element of $\overline{\mathcal F}$.\\
We are going to show that every totally ordered subset in  $(\overline{\mathcal F}, \preceq)$ has  an upper bound  in $\overline{\mathcal F}$. Then, by Zorn's Lemma, $\overline{\mathcal F}$ contains at least one maximal element for the order relation $\preceq$. Any representative  in $\mathcal F$ of this element will be a \emph{maximal conformally flat extension} of $M$.\\
Let $\{[M_i,\psi_i, f_i] \}_{i\in I}$ be a totally ordered subset of $\overline{\mathcal F}$. We can assume, without loss of generality, that $I$ contains a minimal element, denoted by $0$, such that $M_0=M$, $f_0=Id$ and $\psi_0=\phi$. If $i<j$, let  $h_{i,j}: M_i\rightarrow M_j$ be the unique Cauchy-embedding such that $h_{i,j}\circ\psi_i=\psi_j$.    By Lemma \ref{primo},  $h_{i,k}\circ h_{k,j}=h_{i,j}$ and $h_{i,i}=id$ for all $i\leq j\leq k$; moreover $h_{0,i}=f_i$, for all $i\in I$.\\
Let  $$\mathcal{M}:=\bigsqcup_{i\in I}M_i.$$
We consider the following relation on $\mathcal{M}$:  given  $x\in M_i$ and  $y\in M_j$ then 

\begin{displaymath}
 x\sim y  \ \Leftrightarrow \ \left\{ \begin{array}{ll}
i\leq j &  h_{i,j}(x)=y\\
 \textrm{or}\\
j< i & h_{j,i}(y)=x
\end{array} \right.
\end{displaymath}

\noindent The foregoing shows that $\sim$ is an equivalence relation. Let 
$$\overline{M}:=\mathcal{M}/\sim$$
equipped with the quotient topology. We want to show that $\overline{M}$ is an element of $\mathcal{F}$.\\
Let $p_i: M_i\rightarrow \overline{M}$ the composition of the inclusion $M_i\subset \mathcal{M}$ with  projection to the quotient $\pi:\mathcal{M}\rightarrow \overline{M}$. If $i<j$ we have $p_i(M_i)\subset p_j(M_j)$.

\begin{lem}\label{pi homeo}
Every $p_i$ is a homeomorphism onto its image.
\end{lem}

\begin{proof}
By definition of quotient  topology, $p_i$ is continuous. Let  $U\subset M_i$ be an open set, then $p_i(U)$ is open in $\overline{M}$ if and only if $\pi^{-1}(p_i(U))$ is open in $\mathcal{M}$. We have:
$$\pi^{-1}(p_i(M_i))=\left(\bigsqcup_{j\in I:\ i< j}h_{i,j}(M_i)\right)\sqcup \left(\bigsqcup_{j\in I:j\leq i }M_j\right)$$
 Moreover every $h_{i,j}(M_i)$ is open in  $M_j$ because $h_{i,j}$ is an embedding. It is  then clear that  $\pi^{-1}(p_i(U))$ is open in $\mathcal{M}$.
\end{proof}

  The delicate point in the proof that $\overline{M}$ lies in $\mathcal{F}$ is to show is that $\overline{M}$ is a manifold, in particular that it is a second-countable topological space. This is not trivial because $I$ is not countable in general.


\begin{prop}\label{max espace-temps}
$\overline{M}$ is a conformally flat space-time.
\end{prop}

\begin{proof}
According to Lemma \ref{pi homeo}, every point $p$ of $\overline{M}$ is contained in a neighborhood homeomorphic to $M_i$ for some $i$, hence $p$ has a neighborhood homeomorphic to  $\mathbb{R}^n$. Moreover, every pair of points $q_1, q_2$ of $\overline{M}$ is contained in the same $p_i(M_i)$ for some $i\in I$. By Lemma \ref{pi homeo}, $q_1, q_2$ have two disjoint neighborhoods, so the space $\overline{M}$ is a Hausdorff topological space. To conclude the proof  we have to show that  $\overline{M}$ is second-countable and that it is endowed with a conformally flat Lorentzian metric.\\
We first consider the case where $M$ is simply connected.
According to Theorem \ref{teo: Geroch}, the topology of any   globally hyperbolic space-time is determined by the topology of its Cauchy  hypersurfaces. Hence, in our case, every space-time in $\mathcal{F}$ is simply connected. Therefore for every  $M_i$ in  $\mathcal{F}$ there is a developing map $d_i :M_i\rightarrow Ein_{1,n}$.\\
Let $i<j$ be two elements of $I$. The map $d_j\circ h_{i,j}: M_i\to \E_{1,n}$ is another developing map for $M_i$. Therefore there is a unique $g_{i,j}$ in $O(2,n)$ such that $d_j\circ h_{i,j}=g_{i,j}\circ d_i$. We can then define a map $\overline{d}: \overline{M}\rightarrow Ein_{1,n}$  by:

\begin{displaymath}
\overline{d}(x)= \left\{ \begin{array}{ll}
d_0\circ p_0^{-1}(x) &  \text{\ if\ }x\in p_0(M) \\
 \textrm{\ }\\
(g_{0,i})^{-1} \circ d_i\circ p_i^{-1} (x) & \text{\ if\ }x\in p_i(M_i)
\end{array} \right.
\end{displaymath}

\noindent First we have to show that $\overline{d}$ is a well-defined map. Let $x$ be a point in $\overline{M}$; $x$ is contained in  $p_i(M_i)$ for some $i$. Let $j$ such that  $i<j$; then $p_i(M_i)\subset  p_j(M_j)$. Since $h_{0,j}=h_{i,j}\circ h_{0,i}$ we obtain $g_{0,j}=g_{i,j}\circ g_{0,i}$. Then
\begin{align*}
(g_{0,j})^{-1} \circ d_j\circ p_j^{-1} (x)= &(g_{i,j}\circ g_{0,i})^{-1} \circ d_j\circ p_j^{-1} (x)\\
=&(g_{0,i})^{-1}\circ(g_{i,j})^{-1} \circ d_j\circ p_j^{-1} (x)\\
=&(g_{0,i})^{-1}\circ(g_{i,j})^{-1} \circ d_j\circ h_{i,j}\circ (p_i)^{-1} (x)\\
=&(g_{0,i})^{-1} \circ d_i\circ p_i^{-1} (x)
\end{align*}
\noindent the last equality being true because $p_j=(h_{i,j})^{-1}\circ p_i$ and $(g_{i,j})^{-1}\circ d_j\circ h_{i,j}=d_i$. \\
 The map $\overline{d}$ is well-defined and, by construction,  a local homeomorphism.\\
The pull-back  by  $\overline{d}$ of any Riemannian metric over  $Ein_{1,n}$ defines a Riemannian metric over $\overline{M}$: then the open balls for this metric on $\overline{M}$ give a countable basis for the topology, so $\overline{M}$ is a second-countable topological space.  Moreover the map $\overline{d}$ defines a conformally flat Lorentzian structure on $\overline{M}$. Since the map $h_{i,j}$ preserves the orientation and the chronological orientation,  the map $g_{j,i}$ is an element of $O_0(2,n)$, for all $i\leq j$. It implies that $\overline{M}$ is  chronologically oriented. We have proved that, when $M$ is simply connected, $\overline{M}$ is a conformally flat space-time.\\
We can now show the theorem in the general case, when $M$ is not necessarily simply connected.
First we prove that the universal covering  $\tilde{M}$  of $M$ has a naturally defined space-time structure: the lifting by the covering map $\pi: \M\to M$ of the causal structure of $M$.

\begin{lem}\label{lemma: releve dell'ipersuperficie di Cauchy}
Let $M$ be a globally  hyperbolic space-time and let $S$ be a Cauchy hypersurface of $M$. Let $\M$ be the  universal covering of $M$. Then every lift $\tilde{S}$ of $S$ is a Cauchy hypersurface of $\M$.
\end{lem}

\begin{proof}
The covering map $\pi: \M \to M$ is a local diffeomorphism which preserves the causal structures of $\M$ and $M$. In particular, causal curves in $\M$ are precisely lifts of causal curves in $M$. It follows that $\tilde{S}$ is a locally achronal embedded hypersurface of $\M$. If $c$ is a timelike curve intersecting $\tilde{S}$ twice, then the projection $\pi\circ c$ intersects $S$ twice: it is impossible since $S$ is achronal; therefore, $\tilde S$ is also achronal in $\M$. Moreover $\tilde{S}$ is edgeless because this is a local property.\\
Let $\alpha:\R\to \M$ an inextensible causal curve of $\M$.  The map $\pi\circ \alpha$ is an immersion such that the image of every  vector which is tangent to $\alpha$ is a causal vector of $M$. By Definition \ref{def: GH diamante compatto}, $M$ is strongly causal, then $\pi\circ \alpha$ must  be injective and not self-accumulating. This means that $\pi\circ \alpha$ is an embedding. The map $\pi\circ \alpha$ is then an inextensible causal curve of $M$. Since $S$ is a Cauchy hypersurface of $M$, the curve $\pi\circ \alpha$ intersects $S$. This implies that  $\alpha$  intersects $\tilde{S}$, hence $\tilde{S}$ is edgeless and a Cauchy hypersurface of $\M$.
\end{proof}

\noindent Let $\tilde{h}_{i,j}:\tilde{M_i}\rightarrow \tilde{M_j}$ be the conformal embedding which lifts the map  $h_{i,j}$, where  $\tilde{M_i}$, $\tilde{M_j}$ are the universal coverings of $M_i,M_j \in\mathcal{F}$ with $i<j$. According to  Lemma  \ref{lemma: releve dell'ipersuperficie di Cauchy} the lift of every Cauchy hypersurface $S_i$ of $M_i$ is a Cauchy hypersurface  $\tilde{S}_i$ of $\tilde{M}_i$. Therefore, the maps $h_{i,j}$ are conformal Cauchy-embeddings. The following diagram commutes: 

$$\tilde{\Sigma}\xrightarrow{\tilde{\psi_0}}\tilde{M_0}\xrightarrow{\tilde{h}_{0,i}} \tilde{M_i}\xrightarrow{\tilde{h}_{i,j}} \tilde{M_j}\longrightarrow \cdots$$
$$ \ \ \ \ \ \ \ \ \  q_0 \downarrow  \ \ \ \ \ \ \ \ \ q_i \downarrow \ \ \ \ \ \ \ q_j \downarrow\ \ \ \ \ \ \ \ \ \ \ \  $$
$$\Sigma \xrightarrow{\psi_0}M_0\xrightarrow{h_{0,i}} M_i\xrightarrow{h_{i,j}} M_j\longrightarrow \cdots$$

\noindent where $q_k$ are the covering maps. By the same process used in the definition of  $\overline{M}$ we can define  a space-time $\overline{\tilde{M}}$ which now is second-contable and equipped with a  naturally defined conformally flat space-time structure. Let  $\overline{d}$ be the developing map of $\overline{\tilde{M}}$ and let $\tilde{p}_i:\tilde{M}_i\to \overline{\tilde{M}}$ be the continuous and open maps given by Lemma \ref{pi homeo}.
We define the map $p:\overline{\tilde{M}}\rightarrow\overline{M}$ as
$$p(x)=p_i\circ q_i\circ \tilde{p}_i^{-1}(x)$$ 
where $x\in \tilde{p}_i(\tilde{M}_i)$. 
 This definition is independent to the choice of the map $\tilde{p}_i$. Indeed, if $i<j$,  for every $x$ in $\tilde{p}_i(\tilde{M}_i)$  we have $\tilde{p}_i=\tilde{p}_j\circ\tilde{h}_{i,j}$ and $p_i\circ q_i=p_j\circ q_j\circ \tilde{h}_{i,j}$, which implies:
$$p_i\circ q_i\circ \tilde{p}_i^{-1}(x)=p_j\circ q_j\circ \tilde{p}_j^{-1}(x).$$
The map $p$ is a local diffeomorphism since it is the composition of local diffeomorphisms. We want to show that $p$ is a covering map.
Let $\Gamma:=\pi_1(\Sigma)$. The group $\Gamma$ acts over  $\tilde{M_i}$ in such a way that for every $\gamma$ in $\Gamma$, $q_i \circ \gamma=q_i$ and for every $i,j$ in $I$  $\tilde{h}_{i,j}\circ\gamma=\gamma\circ \tilde{h}_{i,j}$. Then we can define an action of $\Gamma$ over $\overline{\tilde{M}}$ by:
$$ \gamma(x) :=\tilde{p}_i\circ\gamma\circ \tilde{p}_i^{-1}(x) $$
where $\forall x \in \tilde{p}_i(\tilde{M}_i)$. This action is well-defined: if $i<j$, then
$$\tilde{p}_i\circ\gamma \circ \tilde{p}_i^{-1}(x)=\tilde{p}_j\circ\tilde{h}_{i,j}\circ\gamma \circ \tilde{h}_{i,j}^{-1}\circ \tilde{p}_j^{-1}(x)=\tilde{p}_j\circ\gamma \circ \tilde{p}_j^{-1}(x).$$
  By construction, $p\circ\gamma=p$. Moreover, for every $x,y$ in $\overline{\tilde{M}}$ we have $p(x)= p(y)$ if and only if there is an element $\gamma$ of $\Gamma$ such that $x=\gamma(y)$.
This action is proper and discontinuous since $\Gamma$ acts properly and discontinuously over every $\tilde{p}_i(\tilde{M}_i)$. Then $\overline{M}$ is the quotient of $\overline{\tilde{M}}$ by $\Gamma$ and $p$ is the projection to the quotient.\\
Since $\overline{M}$ is the quotient of a second-countable manifold by a  proper and discontinuous action, it is also a second-countable manifold. Moreover, since the maps $d_i$ are equivariant for the action of $\Gamma$ over $\tilde{M}_i$, the local diffeomorphism  $\overline{d}$ is also equivariant for the action of $\Gamma$.  Then there is a well-defined conformally flat space-time structure over $\overline{M}$.
\end{proof}

\noindent Now that we have shown that $\overline{M}$ is a space-time, we can study his causal structure. We want to prove that it is globally hyperbolic with Cauchy hypersurface $p_0(S_0)$.

\begin{lem}
Every $p_i(M_i)$ is causally convex inside $\overline{M}$.
\end{lem}

\begin{proof}
Let $c: [0,1]\rightarrow \overline{M}$ be a causal curve between two points of $p_i(M_i)$. The set $c([0,1])$ is a compact subset of $\overline{M}$. Since $\overline{M}$ is the growing union of the sets $p_k(M_k)$ and since $I$ is a totally ordered set,  there exists  $j\in I$ such that $c([0,1])\subset p_j(M_j)$.  We suppose $i\leq j$ (up to  replacing $j$ by $i$ if $j<i$). The open set $p_i(M_i)$ is causally convex in $p_j(M_j)$ since $h_{i,j}(M_i)$ is causally convex in $M_j$ and because of Lemma  \ref{lemma: l'immagine Cauchy-plong. e causalemente convessa}. (Recall that $p_j$ is a conformal embedding.) Thus the image of $c$ is contained in $p_i(M_i)$.
\end{proof}

\noindent The image $p_i(S_i)$ does not depend on $i$, it is a spacelike hypersurface inside $\overline{M}$ that we denote by $\overline{S}$.

\begin{lem}
$\overline{S}$ is a Cauchy hypersurface  of $\overline{M}$. In particular, $\overline{M}$ is globally hyperbolic.
\end{lem}

\begin{proof}
 Let  $c$ be a timelike curve between two points of  $\overline{S}$. Since $c$ is compact, there exists $i\in I$ such that $c$ is contained in $p_i(M_i)$. This contradicts the fact that  $\overline{S}=p_i(S_i)$  is achronal in every $p_i(M_i)$. Thus $\overline{S}$ is achronal.\\
 The hypersurface $\overline{S}$ is edgeless because this is a local property and $\overline{S}$ is edgeless in every $p_i(M_i)$. Every point $p$ of $\overline{M}$ is contained in $p_i(M_i)$ for some $i\in I$, which is globally hyperbolic
 with Cauchy hypersurface $\tilde{S}$. Hence $p$ is contained in the Cauchy development of $\overline{S}$ in $\overline{M}$. This proves that $\overline{S}$ is a Cauchy hypersurface  of $\overline{M}$.
\end{proof}

\noindent We have proved that $\overline{M}$ is an element of $\mathcal{F}$. Moreover, for every $i\in I$,
the map $p_i$ is a Cauchy-embedding of $M_i$ into $\overline{M}$. Let $\bar{f}$ be the Cauchy-embedding given by $p_0$ and $\bar{\psi}$  the composition of $\phi$ and the  restriction of $\bar{f}$ to $S$, then he element $[\overline{M}, \bar{\psi}, \bar{f}]$ of $\overline{\mathcal F}$
and  is  an upper bound  for the totally ordered set $\{[M_i,\psi_i, f_i] \}_{i\in I}$. By  Zorn's Lemma, $\overline{\mathcal F}$ has at least one maximal element. It follows that any representative in $\mathcal F$ of this maximal element is a  maximal conformally flat extension of $M$.\\

  Now we have to prove the uniqueness of the maximal extension. That is: up to conformal diffeomorphism there is a unique maximal element in $\overline{\mathcal F}$.\\
Let $F_1: M_0\rightarrow M_1^{max}$ and $F_2: M_0\rightarrow M_1^{max}$ be two Cauchy-embeddings of $M_0$ in two maximal extensions: we want to prove that $M_1^{max}$ and $M_2^{max}$ are conformally equivalent.\\
Let  $\mathcal H$ be the set of  quadruples $(M,f,g_1,g_2)$ such that $f: M_0\rightarrow M$, $g_1: M\rightarrow M^{max}_1$ and $g_2: M\rightarrow M^{max}_2$ are Cauchy-embeddings, where   $F_i=g_i\circ f$ for $i=1,2$.
Over $\mathcal H$ we define the  relation
\begin{eqnarray*}
 (M,f,g_1,g_2)\preceq(M',f',g'_1,g'_2) \ \iff \exists \Phi:M\rightarrow M' \text{ Cauchy-embedding }\\
 \text{such that } f'=\Phi\circ f \text{ and } g_i=  g_i'\circ\Phi
\end{eqnarray*}
where $i=1,2$. This relation leads to a partial order  over the  quotient $\overline{\mathcal H}$ by the equivalence relation, which identifies two  quadruples
  $(M,f,g_1,g_2)$ and $(M',f',g'_1,g'_2)$ if the Cauchy-embedding $\Phi$ is surjective. The elements of $\overline{\mathcal H}$ are denoted by  $[M,f,g_1,g_2]$. \\
Just like in the proof of the maximal extension's existence, we prove that $(\overline{\mathcal H},\preceq)$ is inductive: given a totally ordered set $\{[M^k,f^k,g^k_1,g^k_2]\}_{k\in I}$ of $\overline{\mathcal H}$, we consider the quotient $\overline{M}$ of the disjoint union of every  $M^k$  the equivalence relation which identifies every $\Phi_{k,l}(x)$ with $x$ if $k < l$, where $\Phi_{k,l}(x)$ is the unique Cauchy-embedding such that $f^l=\Phi_{k,l}\circ f^k \text{ and } g_i^k=  g_i^i\circ\Phi_{k,l} $ .\\
The maps $g_1^k$, $g_2^k$ and $f_k$ are  compatible with this relation and then they induce, at the quotient level,  the Cauchy-embeddings $\overline{g}_i :\overline{M}\rightarrow M^{max}_i$, where $i=1,2$, and $\overline{f}:M_0\to \overline{M}$. As in the existence  proof, one shows that this quotient is a conformally globally hyperbolic space-time, which gives an upper bound  for the set $\{[M^k,f^k,g^k_1,g^k_2]\}_{k\in I}$ (see Figure \ref{fig:unicita}).\\
Once more, Zorn's Lemma implies that $\overline{\mathcal H}$ contains a maximal element, denoted  by $[\overline{M},\overline{f},\overline{g}_1,\overline{g}_2]$.
\begin{figure}
\centering 
\includegraphics[angle=-90,origin=c,height=7cm]{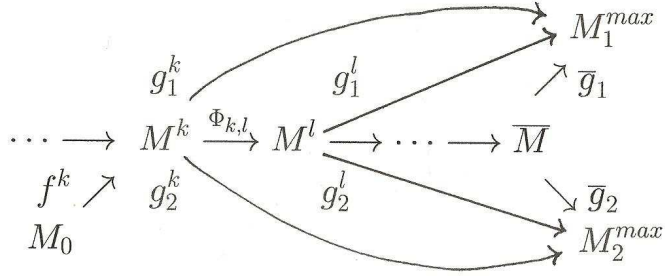}
\caption{ \protect A totally ordered set in $\overline{\mathcal H}$ with the maximal element $[\overline{M},\overline{f},\overline{g}_1,\overline{g}_2]$. \label{fig:unicita}}
\end{figure}
\begin{center}
\end{center}
Let $\mathcal{M}$ be the quotient of the disjoint union $M^{max}_1\bigsqcup M^{max}_2$ by the relation which identifies $\overline{g}_1(x)$ with  $\overline{g}_2(x)$ for every $x\in \overline{M}$. The projections $\pi_i:M^{max}_i\rightarrow \mathcal{M}$ are conformal embeddings, and every point of $\mathcal{M}$ has a neighborhood  homeomorphic to $\mathbb{R}^n$. To prove that  $\mathcal{M}$ is a manifold, the key point that we have to check is:

\begin{lem}
$\mathcal{M}$ is  Hausdorff.
\end{lem}

\begin{proof}
 Let $(x,y)$ be a pair  of points in $M^{max}_1\times  M^{max}_2$ such that every neighborhood of   $\pi_1(x)$ intersects every neighborhood of  $\pi_2(y)$.  
 The points $x$ and $y$ are contained in respectively $\partial \overline{g}_1 (\overline{M})$ and $\partial \overline{g}_2 (\overline{M})$.  Let 
$$S^+_i:=\partial^+ \overline{g}_i (\overline{M}) \text{ and } S^-_i:=\partial^- \overline{g}_i (\overline{M}) $$
 be respectively the past and future boundary  in $M^{max}_i$ of the image of $\overline{M}$ by the embedding  $ \overline{g}_i$,  for $i=1,2$. 
According to Lemma \ref{lemma: frontiera Cauchy-plongement}, $S^+_i$ and $S^-_i$ are closed, achronal, edgeless subsets of $M^{max}_i$, and $S^+_i\cap S^-_i=\emptyset$, for $i=1,2$ . 
Up to time reversal, we can assume, without loss of generality, that
   $x\in  S^+_1$. 
   Let $\Sigma_1$ and $\Sigma_2$ be the Cauchy hypersurfaces respectively of $M^{max}_1$ and $M^{max}_2$ defined by:
   $$\Sigma_1:=\overline{g}_1\circ \overline{f}(S_0) \text{ and } \Sigma_2:=\overline{g}_2\circ \overline{f}(S_0).$$
 Let $U$ be a (connected) neighborhood of $x$ in $M^{max}_1$ such that $U\cap\Sigma_1=\emptyset $ and $U\setminus S^+_1$ is the disjoint union of the two connected open sets : $U\cap \overline{g}_1 (\overline{M}) $ and $U \cap I^+(S^+_1)$. Then $U$ is contained in $I^+(\Sigma_1)$.  The image $U\cap\overline{g}_1 (\overline{M}) $ by the conformal, time-preserving diffeomorphism  $\overline{g}_2\circ \overline{g}_1^{-1}$ is then contained in $I^+(\Sigma_2)$. Thus  $y\in S^+_2$. \\
Since by hypothesis every neighborhood of $\pi_1(x)$ intersects every neighborhood  of $\pi_2(y)$, there exists  a neighborhood $U'$   of $y$ in $M^{max}_2$ such that $\overline{g}_2\circ \overline{g}_1^{-1}(U\cap \overline{g}_1 (\overline{M})  )=U'\cap \overline{g}_2 (\overline{M}) $.\\
   
 \noindent \emph{Key fact: by Liouville's Theorem \ref{Liouville} there exists a  neighborhood $V$ of  $y$ and a conformal diffeomorphism $f: U\rightarrow V$ which is equal to  $\overline{g}_2\circ \overline{g}_1^{-1}$ over
 $U\cap \overline{g}_1 (\overline{M}) $.}\\

\noindent
First assume that the future Cauchy  development  of $D^+_1$ of the achronal set $S^+_1\cap U $ in  $U$ is not empty. Since the open sets $U $ and $V$ are conformally equivalent (by the map $f$), they have the same causal structure. In particular, the future Cauchy  development $D^+_2$ of $S^+_2\cap U_2$ in  $V$  is not empty too. Let 
$$F: \overline{g}_1(\overline{M})\cup D^+_1 \longrightarrow \overline{g}_2(\overline{M})\cup D^+_2$$
be the conformal diffeomorphism defined by:
\begin{displaymath}
F(z):= \left\{ \begin{array}{ll}
\overline{g}_2\circ \overline{g}_1^{-1}(z) &  \text{\ if\ }z\in \overline{g}_1(\overline{M}) \\
 \textrm{\ }\\
f(z) & \text{\ if\ }z\in D^+_1
\end{array} \right.
\end{displaymath}
\noindent Thus, the two element of the set $\overline{\mathcal{H}}$ given by
$$\left[\overline{g}_1(\overline{M})\cup D^+_1,\  \overline{g}_1\circ \overline{f},\  id,\  F\right]  \ \ \text{ and }\  \ \left[\overline{g}_2(\overline{M})\cup D^+_2,\  \overline{g}_2\circ \overline{f},\  F^{-1}, \ id \right]$$
\noindent are  strictly greater then $[\overline{M},\overline{f},\overline{g}_1,\overline{g}_2]$. This contradicts the maximality
 of  $[\overline{M},\overline{f},\overline{g}_1,\overline{g}_2]$.\\

\noindent Therefore, the future Cauchy development of  $S^+_1\cap U $ in  $U$ is empty.\\
According to Lemma \ref{lemma: developpement vide}, there is a lightlike  past  geodesic  $c$ starting from $x$ which is contained in  $S^+_1\cap U $ at least in a neighborhood of $x$.  Let $r_1 : [0, a[ \rightarrow M^{max}_1$ be an inextensible lightlike past geodesic containing $c$. The curve $r_1$ intersects  the  Cauchy hypersurface $\Sigma_0 $ of $M^{max}_1$, in a point $r_1(T_0)$. Since $\Sigma_0$ is contained in $\overline{g}_1(\overline{M})$ and $S^+_1$ is achronal and closed,  there exists a $T\in [0,a[$ such that 
$$T:=\max\{    t \in [0, a[  :  r_1(t)\in S_1^+ \}$$
\noindent   We have  $0<T < T_0 < a$ and $c(]T, T_0])$ is contained in $\overline{g}_1(\overline{M})$. \\
Let $\{U_1\}_{i=0,\dots, m}$ be a finite covering of the compact set $r_1([0,T ])$ such that    $U_i\cap \overline{g}_1 (\overline{M}) $ and $U_i \cap I^+(S^+_1)$ are connected for all $i=1,\dots, m$.
 We can suppose $U_0=U$. Again,  by Liouville's Theorem  there exists a finite sequence $\{V_1\}_{i=0,\dots, m}$  of  open sets and   conformal diffeomorphisms $f_i: U_i\rightarrow V_i$, for $i=0,\dots, m$ such that  $f_0=f$, $V_0=V$ and   $f_i$ is equal to $\overline{g}_2\circ \overline{g}_1^{-1}$ over
 $U_i\cap \overline{g}_1 (\overline{M})$. \\
The isometries $f_i$ glue together to give a map   $F : \bigcup_{i=0}^nU_i\to \bigcup_{i=0}^nV_i$.  Then, for every $t\in [0,T]$, the points   $\pi_1(r_1(t))$ and $ \pi_2(F\circ r_1(t))$ of $\mathcal M$ have no disjointed neighborhoods. \\
Now we can again apply Lemma   \ref{lemma: developpement vide}, to find a past lightlike geodesic  $\gamma$ starting from $r_1(T)$ and contained in $S_1^+$ for a little while. According to Proposition \ref{prop: geodesiche luce e causalità},  $\gamma$ is contained inside $r_1$, because  $S^+_1$ is achronal. But if $\gamma$ is contained in $r_1$ then  $r_1(T+\varepsilon)\in S^+_1$  for some $\varepsilon>0$,   and this contradicts the definition of $T$.  The Lemma is proven. \end{proof}

\noindent Let $\Sigma:=\pi_1\circ \overline{g}_1\circ \overline{f}(S)$. Observe that $\Sigma$ is equal to $\pi_2\circ \overline{g}_2\circ \overline{f}(S)$.

\begin{lem}
\label{le:ghcs}
$\mathcal{M}$ is globally hyperbolic with Cauchy hypersurface $\Sigma$.
\end{lem}

\begin{proof}
Let $c :\mathbb{R}\rightarrow \mathcal{M}$ be an inextensible causal curve. Suppose that $c$ does not intersect  $\pi_1 \circ \overline{g}_1(\overline{M})=\pi_2\circ \overline{g}_2(\overline{M})$.  Then, since $\mathbb{R}$ is connected, $c$ is contained in either $\pi_1(M^{max}_1)\setminus \pi_1\circ \overline{g}_1(\overline{M})$ or in $\pi_2(M^{max}_2)\setminus \pi_1\circ \overline{g}_1(\overline{M})$. Switching the indices if necessary,
we can suppose that $c$ is contained in $\pi_1(M^{max}_1)\setminus \pi_1\circ \overline{g}_1(\overline{M})$. Then, there is a causal curve $c' : \mathbb{R}\rightarrow  M^{max}_1$ such that $c = \pi_1 \circ c'$.
Since $c$ is inextensible, so is $c'$, but $c'$ must intersect $\overline{g}_1(\overline{M})$. This is a contradiction.\\
We have proved that every inextensible causal curve  $c$ intersects $\pi_1 \circ \overline{g}_1(\overline{M})$.
Every connected component  of its intersection with $\pi_1 \circ \overline{g}_1(\overline{M})\simeq \overline{M}$   is an   inextensible causal curve (into $\pi_1 \circ \overline{g}_1(\overline{M})$) and thus intersects  $\pi_1\circ \overline{g}_1\circ f_1(S)=\Sigma$.
Moreover, since $\pi_1 \circ \overline{g}_1(\overline{M})$ is  causally convex inside $M_1^{max}$, it is also causally convex inside $\mathcal{M}$, implying that $c\cap \pi_1 \circ \overline{g}_1(\overline{M})$ has only one connected component. We have proved that every inextensible causal curve intersects  $\Sigma$. The hypersurface $\Sigma$ is achronal and edgeless in $\mathcal{M}$, 
so we have proved that $\Sigma$ is a Cauchy hypersurface of $\mathcal{M}$.
\end{proof}

\noindent As consequence of the proof of Lemma \ref{le:ghcs}, the maps $\pi_i$ are Cauchy-embeddings. The space-time $\mathcal{M}$ is conformally flat because it is covered by conformally flat open subsets given by the images of $\pi_i \circ \overline{g}_i$, $i=1,2$. Since, by hypothesis, the space-times $M_1^{max}$ and $M_2^{max}$ are maximal, the Cauchy-embeddings $\pi_1$  and $\pi_2$ are surjective, and  $M_1^{max}$ and $M_2^{max}$ are conformally equivalent. The proof of Theorem \ref{teo max} is complete. \qed

\begin{oss}\label{oss: extension conforme}
\emph{ Until now we have restricted ourselves to the case of conformally flat space-times.  If  we drop this restriction and define  $\mathcal{F}$  as the set of  triples  $(N,\psi, f)$, where $\psi$ and $f$ are exactly as before, but where $N$ is only globally hyperbolic and not necessarily conformally flat, most of the arguments in the proof above still apply. Given a totally ordered set $\{(M_i,S_i)\}_{i\in I}$ we can still define a topological space $\overline{M}$, candidate to be a \em maximal conformal extension \em  of $M$.\\
 But if one wants to prove that  $\mathcal{M}$ is a manifold, serious troubles arise. The proof of Proposition \ref{max espace-temps} strongly uses the conformally flat structure of the manifolds $M_i$.
 The most delicate point is to prove that $\overline{M}$ has a second-countable basis as a topological space.\\
Even if one were able to prove the existence of the conformal maximal extension, we have no way to prove uniqueness. In the conformally flat case, uniqueness follows from  Liouville's theorem for  conformally flat manifolds of dimension $\geq 3$.\\
However,   we still don't know any examples admitting different maximal conformal extensions. Actually,  all known examples of conformally flat  maximal space-times are also conformal maximal. So it seems interesting to investigate if, despite these difficulties, there is a well-defined notion of maximal conformal extension, in the category of conformal globally hyperbolic space-times.}
\end{oss}

\section{ Complete $C_0$-maximal space-times} \label{sec: thm punti coniugati}

We can easily prove that $\E_{1,n}$ is $C_0$-maximal. The proof goes as follows. Let $f:\E_{1,n}\to N$ be a Cauchy embedding of $\E_{1,n}$. According to Proposition \ref{cor: Geroch produit}, $N$ is simply connected. But, since $N$ is conformally flat, there exists a developing map $D: N\to\E_{1,n}$. The composition $D\circ f:\E_{1,n}\to \E_{1,n}$ preserves the conformal structure of $\E_{1,n}$,  and by Liouville's Theorem $D\circ f $ is an element of $\widetilde{O_0}(2,n+1)$. Hence, $f$ is surjective, and $\E_{1,n}$ is $C_0$-maximal. 
The following lemma is well known in the literature. (See for example \cite{Mess} Lemma 1 p. 7 for an idea of the proof.)

\begin{lem}\label{lem: rivestimento}
Let $f:(S,g)\rightarrow (S',g')$ be a local diffeomorphism between two Riemannian manifolds. Suppose that $(S,g)$ is complete and that $g(\cdot, \cdot)\leq f^*g'(\cdot, \cdot)$. Then $f$ is a covering map.
\end{lem}


\begin{lem} \label{lem: D iniettiva su geodesiche luce}
Let $M$ be a causal, conformally flat, simply connected space-time. Then the developing map $D:M\to\E_{1,n}$ restricted to any causal curve of $M$  is injective.
\end{lem}

 \begin{proof}
 Let $T: \E \to \mathbb{R}$ be any time function on $\E$.
 The restriction of $T \circ D$  to any future causal curve of $M$ is increasing, hence injective. The lemma is proved.
 \end{proof}

\noindent Now we can prove Theorem \ref{teo: compatta s connessa}, which can be restated as follows:

\begin{teo}\label{teo: compatta s connessa2}
For $n\geq 2$, $\E_{1,n}$ is the only conformally flat, simply connected, Cauchy-compact space-time which is $C_0$-maximal.
\end{teo}

\begin{proof}
Let $M$ be a conformally flat simply connected $C_0$-maximal space-time and let $\Sigma\subset M$ be a compact, spacelike Cauchy hypersurface of $M$. According to Corollary \ref{cor: Geroch produit} $M$ is diffeomorphic to the product $\Sigma \times \mathbb{R}$, and $\Sigma$ is simply connected.\\
Let $D: M \rightarrow Ein_{1,n}$ be the developing map of the conformally flat structure on $M$. Consider the decomposition  $Ein_{1,n}\simeq (\mathbb{S}^n\times \mathbb{R},[d\sigma^2-dt^2])$, described in  Section \ref{sezione: Ein}. The pull-back of the metric $d\sigma^2-dt^2$ by $D$ is a Lorentzian metric over $M$ which defines a Riemannian metric $g_0$ on $\Sigma$. The metric $g_0$ is complete  because $\Sigma $ is compact.\\
Let $\pi: Ein_{1,n}\simeq \mathbb{S}^n\times \mathbb{R}\rightarrow \mathbb{S}^n$ be the projection on the first factor. This is a bundle whose fibers are the orbits of the timelike vector field $\partial_t$. The main observation is that $\pi^*d\sigma=d\sigma\geq d\sigma- dt^2$ which implies that $d\pi$ expands the lengths. Since $D\vert_\Sigma$ is a local isometry, $d(\pi \circ D\vert_\Sigma)$ also expands the lengths. Explicitly, for all  $X$ in $T_x\Sigma\subset T_xM$ we have :
$$ \ g_0(X,X)=g (d_xD X,d_xD X)= d\sigma(Y,Y) - \vert a\vert^2 \leq d\sigma(Y,Y)$$
where $Y$ is the component of $X$ which is orthogonal to the orbit of $\partial_t$ and $a$ the tangent one.\\
According to Lemma \ref{lem: rivestimento}, the map $\pi \circ D\vert_\Sigma: \Sigma \rightarrow \mathbb{S}^n$ is a covering map.
Since $\Sf^n$ is simply connected, $\pi \circ D\vert_\Sigma$ is a diffeomorphism.
Therefore the embedded hypersurface $D(\Sigma)$ intersects each orbit of $\partial_t$ at exactly one point. The pull-back on $M$ by $D$ of the vector field  $\partial_t$ of  $Ein_{1,n}$ is a timelike vector field $T$. By Lemma \ref{lem: D iniettiva su geodesiche luce}, the map $D$ is injective along the orbits of  $T$. Since $\Sigma$ is a Cauchy hypersurface of $M$, every orbit of $T$ intersects $\Sigma$ at exactly one point. Let $p,q\in M$, $p\neq q$, such that $D(p)=D(q)$ and let $c_p, c_q$ be the two orbits of $T$, respectively through $p$   and $q$. These orbits intersect $\Sigma$ at two points $\overline{p}$ and $\overline{q}$, obviously we have $\overline{p}\neq \overline{q}$. Since $D$ is a local diffeomorphism, $D(c_p)=D(c_q)$ and it is contained in the orbit of $\partial_t$ through the point $D(p)=D(q)$, which intersects $D(\Sigma)$ at a unique point $x$. Thus $x=D(\overline{p})=D(\overline{q})$. This is absurd because $D$ is injective on $\Sigma$. Therefore, $D$ is injective. \\
Moreover, since the restriction of $D$ to $\Sigma$ is injective, the image $D(\Sigma)$ is a spacelike, compact, edgeless, embedded hypersurface in $\E_{1,n}$.  By Proposition \ref{prop: ipersuperficie de Cauchy compatta}, $D(\Sigma)$ is a Cauchy hypersurface of $Ein_{1,n}$ and $D$ is a Cauchy embedding. Since $M$ is $C_0$-maximal, Theorem  $\ref{teo max}$ implies that $D$ is also surjective: $M$ and $Ein_{1,n}$ are  conformally equivalent. 
\end{proof}

\noindent Let  $M$ be  a conformally flat space-time of dimension $n+1\geq 3$. Let $D: \M\to \E_{1,n}$ be the developing map and $\rho: \pi_1(M)\to O_0(2,n+1)$ the holonomy morphism.  Lemma \ref{lem: D iniettiva su geodesiche luce} is related to the problem of finding injectivity domains for $D$. Let $x\in \E_{1,n}$. By Remark \ref{oss: geodesiche luce E intersezioni}, all the lightlike geodesics starting from $x$ intersect each other only at the conjugate points $\sigma^k(x)$ for every $k\in \Z$. Let $\alpha, \beta : \R\rightarrow \M$, $\alpha(0)=\beta(0)=p$ be two distinct past inextensible lightlike geodesics, such that their images $D\circ \alpha$ and $D\circ \beta$ have an intersection point $z$ other then $D(p)$. Then, if the restriction of $D$ to $\alpha\cup \beta$ is injective,  $\alpha $ and $\beta $ necessarily have an intersection point $q\in D^{-1}(z)$. Conversely, if $D(\alpha)\cap D(\beta)=\emptyset$, then the restriction of $D$ to $\alpha\cup \beta$ is injective.




\begin{teo}\label{1 punto coniugato}
Let $M$ be a globallly hyperbolic conformally flat, $C_0$-maximal space-time of dimension $n+1\geq 3$, and let $D:\M\rightarrow \E_{1,n}$ be the developing map. Suppose that there is a lightlike geodesic $\alpha$ in $\M$ whose image by $D$ contains two conjugate points in $\E_{1,n}$. Then $D$  is a diffeomorphism.
\end{teo}

\begin{proof}
By hypothesis $D\circ \alpha$ contains two conjugate points $x$ and $y$ of $\E_{1,n}$. We can assume, without loss of generality, that  $x=\sigma(y)$.  Since the restriction of $D$ to  $\alpha$ is injective, there are exactly two points   $p$ and $q$  of $\alpha$ such that $D(p)=x$ et $D(q)=y$.\\

\noindent\em Step 1): Every lightlike past  inextensible  geodesic starting from $p$ meets $q$.\em \\
Let $$\mathcal{N}^-(p):=\{ v\in T_p\M / \ v \text{ past lightlike}\ \}$$
On $\mathcal{N}^-(p)$ we define the equivalence relation identifying $v\sim w$ if and only if there exists $\lambda\in \R^+$ such that $v=\lambda w$.
Consider the quotient set $$\mathcal{S}^-(p):=\mathcal{N}^-(p)/ \sim$$
This is the set of lightlike past rays starting from $p$. Endowed with the  quotient topology, it is homeomorphic to the
$(n-1)$-sphere.
Let $p_1:\mathcal{N}^-(p)\rightarrow \mathcal{S}^-(p)$ be the projection map. \\
Let $\mathcal{E}$ the set of  $l$ in $\mathcal{S}^-(p)$ such that the past inextensible lightlike geodesic  $\gamma_l$ starting from $p$ and tangent to $l$, contains $q$.  
By hypothesis, $\mathcal{E}$ is not empty, because it contains the direction tangent to $\alpha$. We are going to show that $\mathcal{E}$ is open and closed in $\mathcal{S}^-(p)$. \\

 \emph{The set $\mathcal{E}$ is open.} \\
Let $l_0$ be an element of $\mathcal{E}$. By definition, $\gamma_{l_0}$ contains $q$.
Let $V$ be an open neighborhood of $y$ and let $W$ be the connected component of $D^{-1}(V)$  containing $q$. We can choose  $V$   such  that $D\vert_W$ is injective.\\
Since $D$ is continuous  and since $\gamma_{l_0}$ is past inextensible,  $D\circ\gamma_{l_0}$ intersects  $J^-(y)\setminus{y}$.  Let $z$ be a point of $D\circ \gamma_{l_0}$ which is in $( J^-(y)\setminus{y})\cap V$ and let $q_0$ be the unique point of $\gamma_{l_0}$ such that $D(q_0)=z$.\\
Since $J^+(y)$ is closed in $V$, we can find an open connected neighborhood  $V'$ of $z$ such that $V'\subset V$ and $V'\cap J^+(y)=\emptyset$. The connected component $W'$ of $D^{-1}(V')$ contained in $W$ is an open neighborhood of $q_0$.\\
Pick a Lorentzian metric $g$ in the conformal class of $\M$ and let $\exp_p: T_p\M\to \M$ be its exponential map at the point $p$.  The set $\mathcal{W}:=\exp_p^{-1}(W')$ is an open subset of $T_pM$. The subset of $S^-(p)$ defined by
$$\mathcal{U}:=p_1(\mathcal{W}\cap \mathcal{N}^-(p))=\{l\in S^-(p)/\ \gamma_l\cap W'\neq \emptyset\} $$
is then open and contains $l_0$. \\
Let $l$ be in $\mathcal{U}$. The curve $\gamma_l$ intersects $W'$, therefore $D\circ \gamma_{l}$ intersects $V'$. It follows that $D\circ \gamma_{l}$ is a past lightlike geodesic starting at $x$ which intersects $J^-(y)$. By Remark \ref{oss: geodesiche luce E intersezioni}, the curve $D\circ \gamma_{l}$ contains $y$.
Since $D\vert_W$ is injective, $\gamma_l$ contains $q$. This shows that  $\mathcal{U}$ is an open neighborhood of $l_0$ contained in $\mathcal{E}$ (see Figure \ref{fig: GraficoDImo}).\\

\emph{The set $\mathcal{E}$ is closed.} \\
Let $\{l_n\}_{n \in \N}$ be a sequence in $\mathcal{E}$ converging to an element $l\in S^-(p)$. According to Lemma \ref{lemma: curva limite causale}, the sequence $\{\gamma_{l_n}\}_{n \in \N}$ has a limit curve which is a past causal curve $c$ between  $p$ and $q$. The fact that $x$ is not temporally related to $y$ implies that $q$ is not temporally related to  $p$ (because the image of timelike curves by $D$ are timelike curves). Proposition \ref{prop: geodesiche luce e causalità} implies that $c$ is a lightlike geodesic. Since $l_n\rightarrow l\in S^-( p)$,  the geodesic  $c$ coincides with  the geodesic  $\gamma_l$.\\

\noindent Since $\mathcal E$ is open and closed in the (connected) topological sphere $\mathcal{S}^-(p)$, it is the entire
$\mathcal{S}^-(p)$. This prove Step $1$.\\
Observe that the similar following statement is true: every future lightlike geodesic starting from $q$ contains $p$.\\

\noindent\emph{ Step 2) $\partial I^-(p)$ is a  compact Cauchy hypersurface in $\M$.}\\
Let $S$ be the union of all lightlike geodesic segments between  $p$ and $q$. By construction, $D(S)=\partial I^-(x)$. 
Let us prove that the restriction of $D$ to $S$ is injective. Let $r$, $s$ be two points of $S$ with $D(r)=D(s)$. Since any point of $S$ is contained in a lightike geodesic starting from $p$ and ending in $q$, and since  $D$ restricted to any causal curve is injective, we have that if $r=p$ (or reps. $r=q$) then $s=r=p$ (or reps. $s=r=q$). Then we can suppose $r,s \in S\setminus \{p,q\}$.  Let $\gamma_l$, $\gamma_{l'}$ be the geodesic segments between $p$ and $q$ containing respectively $r$ and $s$. Then $D(\gamma_l)$ and $D(\gamma_{l'})$ are geodesic segments between $x$ and $y$ having the point $D(r)=D(s)$ in common. Moreover, since $\E_{1,n}$ is causal, this point is different to $x$ and $y$.  Then the Remark \ref{oss: geodesiche luce E intersezioni} implies that the two geodesic $D(\gamma_l)$ and $D(\gamma_{l'})$ are the same. Since $D$ is a local diffeomorphism, it follows that $l=l'$. Since the restriction of $D$ to the causal curve $\gamma_l$ is injective, the equality $r=s$ follows.\\
The set   $S$  is achronal in $M$: if there was a timelike curve between two points of $S$, then the image under $D$ of this curve   would be a timelike curve between two points of the achronal subset $D(S)=\partial I^-(x)$.\\
 Since  $S$ is achronal we have $S\subset \partial I^-(p)$. By Proposition \ref{prop: geodesiche luce e causalità} we have also the inclusion  $\partial I^-( p)\subset S$.  Then $S= \partial I^-( p)$.
Since  $D(S)=\partial I^-(x)$ is compact, it follows that $S$ is also compact.
Moreover, according to Lemma \ref{lemma: frontiera di insieme passato}, the boundary $\partial I^-(p)=S$ of the past subset  $I^-(p)$ is  edgeless.\\
 We have proved that $S$ is a compact, achronal, edgeless subset of $\M$: by Proposition \ref{prop: ipersuperficie de Cauchy compatta},  $S$ is a Cauchy hypersurface of $\M$.\\

\noindent\em  Step 3) \em Since $\M$ admits a compact  Cauchy  hypersurface, we obtain by Theorem \ref{teo: compatta s connessa2} that $\M$ is conformally equivalent to $\E$.
\end{proof} 

\noindent We can now give the proof of Theorem \ref{1 punto coniugato}, which we restate here for the reader's convenience:\\

\noindent\textbf{Theorem 1.9.} \textit{Let $M$ be a conformally flat, globally hyperbolic, $C_0$-maximal space-time  which has two freely homotopic lightlike geodesics, which are distinct but with the same ends. Then $M$ is a finite quotient of $\E_{1,n}$.}

\begin{proof}
Since the   two  geodesic segments are freely homotopic, it is possible to lift them in such a way that the two lifts are two lightlike geodesic segments,  $\alpha $ and $\beta$, with the same ends.  The image of $\alpha$ by the developing map (and also the image of $\beta$) intersects two  conjugate points of  $\E_{1,n}$. By Theorem \ref{1 punto coniugato}, $\M$ is homeomorphic  to $\E_{1,n}$. In particular  the lift $\tilde{S}$ of every Cauchy hypersurface   $S$ of $M$ is homeomorphic to the sphere  $\Sf^n$. The fundamental group of $M$ has to preserve $\tilde{S}$ and it acts properly and discontinuously on $\tilde{S}$. Since  $\tilde{S}$ is compact, it follows that $\pi_1(M)$ is finite.
\end{proof}


\bibliographystyle{plain}
\bibliography{thesis}

\end{document}